\documentclass[12pt, colorinlistoftodos]{amsart}
\usepackage{todonotes}
\usepackage{a4wide}
\usepackage{amssymb}
\usepackage{amsthm}
\usepackage{amsmath}
\usepackage{amscd}
\usepackage{verbatim}
\usepackage{bm}
\usepackage[all]{xy}
\usepackage{hyperref}
 \usepackage{dsfont, xtab, stmaryrd}
\numberwithin{equation}{section}

\theoremstyle{plain}
\newtheorem{theorem}{Theorem}[section]
\newtheorem{corollary}[theorem]{Corollary}
\newtheorem{lemma}[theorem]{Lemma}
\newtheorem{proposition}[theorem]{Proposition}

\theoremstyle{definition}

\newtheorem{remark}[theorem]{Remark}
\newtheorem{example}[theorem]{Example}

\theoremstyle{remark}

\newcommand{\OO}{\mathcal O}
\newcommand{\A}{\mathbb{A}}
\newcommand{\R}{\mathbb{R}}
\newcommand{\G}{\mathbb{G}}
\newcommand{\Q}{\mathbb{Q}}
\newcommand{\Z}{\mathbb{Z}}
\newcommand{\N}{\mathbb{N}}
\newcommand{\C}{\mathbb{C}}

\renewcommand{\H}{\mathbb{H}}

\newcommand{\diag}{\operatorname{diag}}
\newcommand{\ord}{\operatorname{ord}}
\newcommand{\kay}{k}
\newcommand{\GSpin}{\operatorname{GSpin}}
\newcommand{\Gspin}{\operatorname{GSpin}}

\newcommand{\SL}{\operatorname{SL}}
\newcommand{\cha}{\operatorname{Char}}

\newcommand{\pmat}[4]{\begin{pmatrix}
                 #1 & #2\\
                 #3 & #4
\end{pmatrix}}
\newcommand{\smat}[4]{\left(\begin{smallmatrix}
                 #1 & #2\\
                 #3 & #4
\end{smallmatrix}\right)}
\newcommand{\kzxz}[4]{\left(\begin{smallmatrix} #1 & #2 \\ #3 & #4\end{smallmatrix}\right) }

\newcommand{\lp}{\left (}
\newcommand{\rp}{\right )}
\newcommand{\Ac}{{\mathfrak{A}}}
\newcommand{\tAc}{{\tilde{\mathfrak{A}}}}
\newcommand{\Bc}{{\mathfrak{B}}}

\newcommand{\Pc}{{\mathcal{P}}}
\newcommand{\Nc}{{\mathcal{N}}}

\newcommand{\Zb}{\mathbb{Z}}
\newcommand{\Qb}{\mathbb{Q}}
\newcommand{\SO}{{\mathrm{SO}}}
\newcommand{\GL}{{\mathrm{GL}}}
\newcommand{\af}{\mathfrak{a}}
\newcommand{\bfrak}{\mathfrak{b}}

\newcommand{\ef}{\mathfrak{e}}
\newcommand{\ff}{\mathfrak{f}}
\newcommand{\wf}{\mathfrak{w}}
\newcommand{\uf}{\mathfrak{u}}
\newcommand{\vf}{\mathfrak{v}}

\newcommand{\Nm}{{\mathrm{Nm}}}
\newcommand{\Ab}{\mathbb{A}}
\newcommand{\Fb}{\mathbb{F}}
\newcommand{\Nb}{\mathbb{N}}

\newcommand{\Cb}{\mathbb{C}}

\newcommand{\Vb}{\mathbb{V}}

\newcommand{\attr}[1]{\color{red} #1 \color{black}}

\newcommand{\taf}{\tilde{\mathfrak{a}}}
\newcommand{\tb}{\tilde{b}}

\newcommand{\tm}{\tilde{m}}
\newcommand{\tmu}{\tilde{\mu}}
\newcommand{\tn}{\tilde{n}}
\newcommand{\tw}{\tilde{w}}
\newcommand{\tx}{\tilde{x}}
\newcommand{\ty}{\tilde{y}}

\newcommand{\tr}{\operatorname{Tr}}

\newcommand{\norm}{\operatorname{N}}

\newcommand{\dd}{\mathrm{d}}

\newcommand{\zr}{\mathrm{z}}
\newcommand{\Diff}{\mathrm{Diff}}
\newcommand{\disc}{\mathrm{disc}}

\newcommand{\Oc}{\mathcal O}

\newcommand{\cf}{\mathfrak{c}}
\newcommand{\rg}{\rho_{\mathrm{g}}}

\newcommand{\fff}{\operatorname{if }}
\newcommand{\Cl}{\operatorname{Cl}}

\AtBeginDocument{%
 \def\MR#1{}
}
\begin{document}

\title[On a Conejcture of Yui and Zagier II]{On a Conjecture of Yui and Zagier II}

\author[Y.~ Li]{Yingkun Li}
\author[T.~Yang]{Tonghai Yang}
\author[D.~Ye]{Dongxi Ye}

\address{
Max Planck Institute for Mathematics,
    Vivatsgasse 7, 
    D--53111     Bonn,
    Germany}
\email{yingkun@mpim-bonn.mpg.de}

\address{Department of Mathematics, University of Wisconsin Madison, Van Vleck Hall, Madison, WI 53706, USA}
\email{thyang@math.wisc.edu}

\address{
School of Mathematics (Zhuhai), Sun Yat-sen University, Zhuhai 519082, Guangdong,
People's Republic of China}

\email{yedx3@mail.sysu.edu.cn}


\thanks{The first author is supported by the LOEWE research unit USAG,
 the Deutsche Forschungsgemeinschaft (DFG) through the Collaborative Research Centre TRR 326 ``Geometry and Arithmetic of Uniformized Structures'' (project number 444845124)
 and the Heisenberg Program ``Arithmetic of real-analytic automorphic forms''(project number 539345613).
The second author is partially supported by a UW-Madison Mid-Career award. The third author is supported by the Guangdong Basic and Applied Basic Research
Foundation (Grant No. 2024A1515030222).}


\date{\today}

\begin{abstract} Yui and Zagier  made some fascinating  conjectures on the factorization on the norm of  the difference of Weber class invariants $ f(\mathfrak a_1) -  f(\mathfrak a_2)$ based on their calculation in \cite{YZ}.  Here $\mathfrak a_i$ belong  two diferent ideal classes of discrimants $D_i$ in imagainary quadratic fields $\Q(\sqrt{D_i})$.  In  \cite{LY}, we proved these conjectures and their generalizations when  $(D_1, D_2) =1$ using the so-called big CM value formula of Borcherds lifting. In this sequel, we prove the conjectures when  $\Q(\sqrt{D_1}) =\Q(\sqrt{D_2})$ using the so-called small CM value formula. In addition, we  give a precise factorization formula for the resultant of two different Weber class invariant polynomials for distinct orders.
\end{abstract}

\maketitle
 \makeatletter
 \providecommand\@dotsep{5}
 \def\listtodoname{List of Todos}
 \def\listoftodos{\@starttoc{tdo}\listtodoname}
 \makeatother

\allowdisplaybreaks

\section{Introduction}

This is a sequel to \cite{LY}. Recall the three classical Weber functions of level $48$:
\begin{align}\label{eq:Weber}
\mathfrak f(\tau) &:=\zeta_{48}^{-1} \frac{\eta(\frac{\tau+1}2)}{\eta(\tau)}=q^{-\frac{1}{48}} \prod_{n=1}^\infty (1+q^{n-\frac{1}2}), \notag
\\
\mathfrak f_1(\tau) &:=\frac{\eta(\frac{\tau}2)}{\eta(\tau)} = q^{-\frac{1}{48}}\prod_{n=1}^\infty (1-q^{n-\frac{1}2}),
\\
\mathfrak f_2(\tau) &:= \sqrt 2 \frac{\eta(2\tau)}{\eta(\tau)}  = \sqrt 2  q^{\frac{1}{24}} \prod_{n=1}^\infty (1+q^n). \notag
\end{align}
where $\zeta_m :=e^{\frac{2 \pi i}{m}}$ for any $m \in \Nb$. 
For each  $\dd|24$, the vector-valued function
\begin{equation}
  \label{eq:Fd}
  F_\dd(\tau) :=
  \sqrt{2}^{24/\dd}
\begin{pmatrix}
  \ff_2^{-24/\dd}(\tau) \\ \ff_1^{-24/\dd}(\tau) \\ \ff^{-24/\dd}(\tau)
\end{pmatrix}
\in M^!(\varrho_\dd)
\end{equation}
is a weakly holomorphic modular form of $\SL_2(\Z)$ with representation  $\varrho_\dd$

\begin{equation}
  \label{eq:varrhod}
  \varrho_{\dd}(T) :=
  \begin{pmatrix}
    \zeta_{\dd}^{-1} & 0 & 0 \\
0 & 0 & \zeta_{2\dd}\\
0 & \zeta_{2\dd}  & 0
  \end{pmatrix},~
  \varrho_{\dd}(S) :=
  \begin{pmatrix}
0 & 1 & 0 \\
1 & 0 & 0 \\
0 & 0 & 1
  \end{pmatrix}.
\end{equation}

The starting point of this project was the following inspiring result of Yui and Zagier.

\begin{proposition}[\cite{YZ} Proposition]
\label{prop:YZ}
Let $D<0$ be a discriminant satisfying
\begin{equation} \label{eqd}
D \equiv 1 \bmod 8,  \hbox{  and }  3\nmid D.
\end{equation}
Denote $\varepsilon_D := (-1)^{(D-1)/8}$.
For each ideal $\mathfrak a=[a, \frac{b+\sqrt D}2]$ of the order $\mathcal O_D := \Zb[\tfrac{D+ \sqrt{D}}{2} ]$ with $a>0$, let $z_{\mathfrak a} =\frac{-b+\sqrt{D}}{2a}$  
be the associated CM point and
\begin{equation}
  \label{eq:classinv}
f(\mathfrak a) =\begin{cases}
 \zeta_{48}^{b(a-c-ac^2)}  \mathfrak f(z_{\mathfrak a} ), &\fff 2|(a, c),
 \\
  \varepsilon_D \zeta_{48}^{b(a-c-ac^2) } \mathfrak f_1(z_{\mathfrak a} ), &\fff 2|a, 2\nmid c,
  \\
  \varepsilon_D \zeta_{48}^{b(a-c+a^2c) } \mathfrak f_2(z_{\mathfrak a} ), &\fff 2\nmid a, 2|c.
  \end{cases}
\end{equation}
Then $f(\mathfrak a)$ is an algebraic integer depending only on the class of $\mathfrak a$ in the class group $\mathrm{Cl}(D)$ of $\mathcal{O}_D$, i.e.\ it is a class invariant. Moreover, $H_D := \kay_D(f(\mathfrak a)) = \kay_D(j(z_\mathfrak{a}))$ is the ring class field of $\kay_D$ corresponding $\mathcal{O}_D$.
\end{proposition}

\begin{remark}
  \label{rmk:fA}
  If $\Ac$ is the class of $\af$ in $\Cl(D)$, then we will also write $f(\Ac)$ for $f(\af)$.
\end{remark}

\begin{remark}
    It was proved in \cite[Theorem 1.1(4)]{YYY} that $f(\mathfrak{a})$ is a unit.
\end{remark}

So for a pair of ideal classes $[\mathfrak a_i]$  of order $\OO_{D_i}$, the difference $f(\mathfrak a_1) -f(\mathfrak a_2)$ is an algebraic number in the field $H_{D_1} H_{D_2}$, and a natural question is what we can say about the factorization of its norm, considering  the beautiful results of Gross and Zagier on singular moduli \cite{GZ}. Yui and Zagier gave two conjectured formula for the norm: one when  $(D_1, D_2) =1$ and the other when  $D_1=D_2$.  In \cite {LY}, two of us  proved the conjecture in the case $(D_1, D_2)=1$ using Borcherds products and big CM value formula of Bruinier-Kudla-Yang \cite{BKY}. In this paper, we will use the \textit{small CM value formula} of Schofer \cite{Schofer} to prove the conjecture in the case $D_1=D_2$. Actually, we will give an explicit formula in the more general case when $D_1/D_2$ is a rational square, i.e.,  $D_i=D_0 t_i^2$ with $(t_1, t_2) =1$ (see Theorem ~\ref{thm:main} for detail).

For a positive integer $n$, let $\rho(n)$ be the number of integral ideals of the quadratic field $\kay_D=\Q(\sqrt{D})$ with norm $n$. It has factorization  
$$
\rho(n) =\prod_{p < \infty} \rho_p(n)
= \sum_{\Ac \in \Cl(\kay_D)} r_\Ac(n)
$$
with
 $r_{\Ac}(n) = \#\{\cf \subset \Oc_D: \Nm(\cf) = n, [\cf] = \tAc\}$  the ideal counting function for $n \ge 1$, and 
\begin{equation}
\rho_p(n) = \begin{cases}
      o_p(n) +1 &\fff p \hbox{ is split},
    \\
     \frac{1 + (-1)^{o_p(n)}}2  &\fff p \hbox{ is inert},
     \\
     1  &\fff p \hbox{ is ramified or } p = \infty.
\end{cases}
\end{equation} 
For convenience, we set $r_\Ac(0) = \frac12, \, \rho(0) = \frac{|\Cl(\kay_D)|}2,  \, \rho_p(0) = 1$, \,  and let 
\begin{equation}  \label{eq:rhom}
  \rho^{(M)}(n) =
  \begin{cases}
    \prod_{ p\nmid M} \rho_p(n) & n > 0,\\
    \rho(0) & n = 0.
  \end{cases}
\end{equation}

The following result confirms the Yui-Zagier conjecture for discriminants (\cite[p. 1658]{YZ}). 
\begin{theorem}\label{corYZ}
Let $D< 0$ be a discriminant satisfying~\eqref{eqd}, and for $s \mid 24$ 
  and each non-trivial class $\tAc=[\tilde{\mathfrak{a}}] \in \Cl(D)$ with ${\rm Nm}(\tilde{\mathfrak{a}})=a$,  denote
  $$
\disc(D; s, \tAc) := \prod_{\Ac \in \Cl(D)} (f(\Ac)^{24/s} -  f(\Ac\tAc)^{24/s}) \in \kay_{D}.
$$
If $-D = p$ is prime, then the followings hold.
\begin{enumerate}
    \item For each prime $\ell$ split in $k_{D}$, we have ${\rm ord}_{\ell}({\rm disc}(D;s,\tilde{\mathfrak{A}}))=0$.
 \item  For each prime $\ell\ne3$ inert in $\kay_{D}$, we have that
$\ord_\ell(\disc(D; s, \tAc))$ 
 amounts to the number of pairs of integral ideals $(\mathfrak{b}_{1}, \bfrak_{2})$ such that
 $[\mathfrak{b}_{2}]=\tilde{\mathfrak{A}}$, $\ell^{j}{\rm Nm}(\mathfrak{b}_{1})+{\rm Nm}(\mathfrak{b}_{2})=p$ for some $j>0$, and that $\mathfrak{a} =(2s')^{-1}\mathfrak{b}_{1}\mathfrak{b}_{2}$ is a  nonzero integral ideal $\mathfrak{a}$ satisfying that
    $$
    c(\mathfrak{a})\left(\ell\frac{{\rm Nm}(\mathfrak{a})}{c(\mathfrak{a})^{2}}+5\right)\equiv0\pmod{\frac{s}{s'}},
    $$
    where $s'=\gcd\left(s,3^{1-\left(\frac{D}{3}\right)}\right)$.

\item For $\ell=3$ inert in $k_{D}$, we have that
\begin{align*}
&{\rm ord}_{3}({\rm disc}(D;s,\tilde{\mathfrak{A}}))\\
   & =   \sum_{\substack{n,\tilde{n}>0\\n + \tn =p\\2\nmid o_{3}(n)\geq1}}
     \sum_{\substack{r \mid (s/s_3),~ r>0\\ \frac{n\tn}{4r^2} \equiv 19 \bmod{\frac{s/s_3}{r}}}}
  \sum_{AB = 2r}
  \rho^{(3)}\lp \frac{n}{A^2}\rp
 r_{\tAc}\lp \frac{\tn}{B^2} \rp
    \frac{o_3(n/s_3) + 1}{2}.
\end{align*}

\item For $\ell=p$, we have that ${\rm ord}_{\ell}({\rm disc}(D;s,\tilde{\mathfrak{A}}))=\frac{1}{2}$.

\end{enumerate}


\end{theorem}

\begin{remark}
The more general case of $D=D_0=D_1=D_2$ being fundamental is given as Theorem~\ref{thm:YZconj}.  
\end{remark}

As another consequence of the main formula (Theorem ~ \ref{thm:main}), we will also prove Theorem \ref{thm:main2a} concerning factorization of the resultant of two class polynomials when $D_1 \neq D_2$. 
Its simplified version is given below.

\begin{theorem}
  \label{thm:main2}
  Let $D_i = D_0 t_i^2$ be discriminants satisfying \eqref{eqd} with $t_1, t_2$ co-prime and $D_0$ fundamental.
    Suppose  all primes dividing $t := t_1t_2 > 1$ are inert in $\kay$. 
    Let $P_i(x)$ be the minimal polynomial of $f(\mathfrak  A_i)$  for $\mathfrak A_i \in \Cl(D_i)$, which depends only on $D_i$, and $R(P_1, P_2)$ be their resultant.  Then for an inert prime $\ell \nmid 3t$,     $\hbox{ord}_{\ell} R(P_1, P_2)$ is given by
\begin{equation}
    \label{eq:res}
          \sum_{\substack{n,\tn\in\mathbb{N}\\n + \tn = -D_0 t\\ \gcd(n, t) = 1}}
\sigma(\gcd(n, {D_0}))
     \sum_{\substack{r \mid s,~ r>0\\s'|r\\ \frac{n\tn}{4r^2} \equiv 19 \bmod{\frac{s}{r}}}}
  \sum_{\substack{AB = 2r\\j \ge 0}}
  \rho^{}\lp \frac{n}{A^2\ell^j}\rp
\rg(\tn/B^2; - n).
\end{equation}
Here $s'=\gcd\left(s,3^{1-\left(\frac{D}{3}\right)}\right)$, and $\sigma(n)=\sum_{d|n}1$ is the divisor function and for $m \ge 1$
$$
\rg(m; c) :=
\sum_{\Ac \in \Cl(D_0), \text{ the genus of }\Ac \text{ represent }c} r_\Ac(m).
$$
\end{theorem}


The current paper improves upon \cite{LY} in the following ways.
By systematically working with the character $\chi$ defined by the Weber function in \eqref{eq:chi}, we clarify the level of the CM points without resorting to explicit computations as in Lemma 5.2 in \cite{LY}. The small CM point formula requires us to investigate the rational splitting the lattice $L$ in \eqref{eq:Ld}. For this, we give a global splitting result in Proposition \ref{prop:Lattice}. Using this, we investigate the local splitting behavior in the last section. Finally, we prove results, such as Theorem \ref{thm:main2a}, for non-maximal orders by giving new formula for values of local Whittaker function in Propositions \ref{prop4.5}, \ref{prop4.6} and \ref{prop4.7}, and specializing them to the cases we need in section \ref{subsec:pmidD}.
Also, the constant term of the incoherent Eisenstein contributes to the final formula, which creates subtlety in Theorem \ref{thm:main2a} that is not present in \cite{LY}.

The  paper is organized as follows. In Section 2, after  reviewing  the Weil representation, we study small CM  points on product of modular curves carefully, including their associated lattices and various identifications. These careful identifications are critical to establishing our main formulas in Section 4.   In Section 3, we recall the main results of Borcherds liftings in \cite{LY} and then the small CM value formula of Schofer (\cite{Schofer}), and using these we establish a preliminary version of our main formula. In Section 4, we do necessary local calculation and establish a general factorization formula for the norm of difference  of Weber functions at a small CM point (Theorem \ref{thm:main}). Using the main theorem, we prove Theorem \ref{thm:YZconj} and  Theorem \ref{thm:main2a} with more careful calculation in these special cases. We also give some explicit examples for these results. 

{\bf Acknowledgement} Later
\section{Preliminaries}
\subsection{Weil Representation.}
Let $(L, Q)$ be an even integral lattice of signature $(2, 2)$, $V := L \otimes \Qb$ the rational quadratic space, and $L' \subset V$ the dual lattice.
For $R = \Ab_f$ or $\Q_p$, denote $S(V\otimes R)$ the space of Schwartz functions on $V \otimes R$, which is acted on by $\SL_2(R)$ via the Weil representation $\omega = \omega_{V, \psi} = \otimes_p \omega_p$ for $R = \Ab_f$ (with $\psi$ the usual idelic character of $\Qb$) and $\omega_p$ for $R = \Q_p$.
The restriction of $\omega$ to the (the diagonally embedded) subgroup $\SL_2(\Zb) \subset \SL_2(\Ab)$ acts on the finite dimensional subspace
$$
S(L):=\oplus_{\mu \in L'/L} \Cb \phi_\mu \subset S(V \otimes \A_f)
$$
with $\hat{L} := L \otimes \hat{\Zb}$ and $\phi_\mu:=\cha(\mu+\hat L)$.
We denote this subrepresentation by $\omega_L$ and also call it the Weil representation associated to $L$.
Its formula on the following standard generators of $\SL_2(\Z)$
$$
T := \pmat{1}{1}{0}{1},~
S := \pmat{0}{-1}{1}{0}
$$
can be explicitly given (see e.g.\ (2.1) of \cite{LY}).
Similarly, $\SL_2(\Z_p)$ acts on $S(L \otimes \Z_p)$ via the local Weil representation $\omega_{L, p}$.

\subsection{Modular curves and product of modular curves as Shimura varieties} Let $\dd>0$ be a fixed positive integer. Let $V=V_\dd=M_2(\Q)$ with quadratic form $Q_\dd(x) =\dd \det x$, and associated bilinear form $(\cdot,\cdot)_\dd$.
Let
$$
H=\Gspin(V) = \{ (g_1, g_2) \in \GL_2 \times \GL_2:\, \det g_1 =\det g_2\}
$$
act on  $V$ via $(g_1, g_2) X = g_1 X g_2^{-1}$. Let $\mathbb D$ be the associated  Hermitian symmetric domain of oriented negative $2$-planes in $V_\R$. It has two connected components $\mathbb D^\pm$. Let $V^0=V_\dd^0$ be the subspace of $V$ of trace zero matrices, $H^0=\GSpin(V^0) = \GL_2$, and $\mathbb{D}^0= \cup \mathbb{D}^{0, \pm}$ the associated oriented negative planes in $V_\R^0$. Then  $H_0 \subset H$ diagonally, and it acts on $V^0$ via conjugation. The following is a special case of \cite[Proposition 3.1]{YY19}

\begin{proposition} Define
$$
 w_{\dd}(z_1, z_2) =\kzxz {\frac{z_1}{\dd}} {\frac{-z_1 z_2}{\dd}} {\frac{1}{\dd}} {\frac{-z_2}{\dd}}.
$$
Then  the map
$$
\H^2 \cup (\H^-)^2 \cong \mathbb D,  \quad (z_1, z_2) \mapsto  \R \Re(w_\dd(z_1, z_2)) + \R \Im (w_\dd(z_1, z_2))
$$
is an isomorphism.  Moreover, $w_{\dd}$ is $H(\R)$-equivariant, where $H(\R)$  acts on $\H^2 \cup (\H^-)^2$ via  the usual linear fraction transformation:
$$
(g_1, g_2)(z_1, z_2) =(g_1(z_1), g_2(z_2)),
$$
and acts on  $\mathbb D$ naturally via its action on $V$. Moreover, one has
\begin{equation} \label{eq:linebundle}
(g_1, g_2) w_{\dd}(z_1, z_2) = \frac{j(g_1, z_1) j(g_2, z_2)}{\nu(g_1, g_2)} w_{\dd}(g_1(z_1), g_2(z_2)),
\end{equation}
where $\nu(g_1, g_2) =\det g_1 =\det g_2$ is the spin character of $H\cong\Gspin(V)$, and
$
 j(g_i, z_i)=c_i z_i+d_i
$
is the automorphy factor (of weight $(1, 1)$).

 Finally, when restricting to $\H$ diagonally, we obtain a $H^0(\R)$-invariant map
 $$
 w_\dd^0(z) = \frac{1}{\dd} \kzxz {z} {-z^2} {1} {-z},
 $$
 and $\H^\pm \cong \mathbb D^{0, \pm}$.
\end{proposition}

\subsection{CM points on $\H$ as CM points associated to $V^0$} Let $z =\frac{b+\sqrt{D}}{2a} \in  \H$ be a CM point by imaginary quadratic order $\OO_D=\Z + \Z \frac{D+\sqrt D}2$, i.e. $\mathfrak a =[a, \frac{b+\sqrt D}2]$ is a integral ideal of $\OO_D$, or equivalently, the elliptic curve $E_z= \C/(\Z + \Z z)$ has endomorphism ring $\OO_D$. Let  $U^0_\R(z) =
\R \Re w^0_\dd(z) + \R \Im w^0_\dd(z)$ be the associated negative two plane, i.e., $U^0_\R(z) = U^0(z) \otimes \R$ with
$$
U^0(z) = \Q \kzxz {\frac{b}{2a}}  { -\frac{D+b^2}{4a^2}} {1} {-\frac{b}{2a}}  +
    \Q \kzxz  {\frac{1}{2a}} {-\frac{b}{2a^2} } {0} {-\frac{1}{2a}}.
$$
Let $T=\kay_D^\times = \Gspin(U^0)$, then we have the commutative diagram
\begin{equation} \label{eq:torus1}
\xymatrix{
 1 \ar[r]  &\G_m \ar[r] \ar[d] &T=\kay_D^\times \ar[r] \ar[d] &\SO(U^0)=\kay_D^1\ar[r] \ar[d] &1
\cr
 1 \ar[r]  &\G_m \ar[r] &H^0=\GL_2(\Q) \ar[r] &\SO(V^0) \ar[r] &1.
 \cr
}
\end{equation}

\begin{lemma} \label{lem:Embedding1} Define $\iota_z: \kay_D\rightarrow M_2(\Q)$ via the $\Q$-basis $\{1, -\overline{z}\}$
as follows
\begin{equation}
  \label{eq:iotaz}
(r,  -\overline{z}r) = (1, -\overline{z}) \iota_z(r).
\end{equation}
Then $\iota_z$ gives the map from $T$ to $H^0$ in the diagram (\ref{eq:torus1}).
\end{lemma}

\begin{proof} Let $z = \frac{b + \sqrt{D}}{2a}$ and
  \begin{equation}
    \label{eq:iotaz2}
P(z) = U^0(z)^\perp = \Q \iota_z(\sqrt D), \quad  \iota_z(\sqrt D) = \kzxz {b} {\frac{D-b^2}{2a}} {2a} {-b}.
  \end{equation}
Then it is easy to see (\cite[Theorem 24.6(iv)]{Sh}) that
$$
T= \Gspin(U^0) =\{ g \in \GL_2(\Q):  g \iota_z(\sqrt D) g^{-1} = \iota_z(\sqrt D)\} = (\Q +\Q \iota_z(\sqrt D))^\times.
$$
So $\iota_z$ gives the embedding from  $T$ to $H^0$.
\end{proof}

\begin{remark}
  \label{rmk:iotaz}
The definition for $\iota_z(r)$ above is equivalent to
\begin{equation}
  \label{eq:iotazp}
  \iota_z(r) \binom{z}{1} = \binom{rz}{r}
\end{equation}
for all $r \in \kay_D$.
Furthermore, we have
\begin{equation}
  \label{eq:iotaequiv}
\gamma^{-1} \iota_{\gamma z}(r) \gamma = \iota_z(r)
\end{equation}
for all $r \in \kay_D, \gamma \in \SL_2(\Z)$.
\end{remark}

\subsection{Small CM points on $\H \times \H$}
\label{subsec:smallCM}
A CM point on $\H^2$ is a pair  $(z_1, z_2) \in \H^2$ such that $z_i$ is a CM point on $\H$ with CM by $\OO_{D_i}$. It is called a small   CM point if $\kay_{D_1} =\kay_{D_2}$ and a big CM point otherwise.  We deal with small CM points here and refer the reader to \cite{YY19} for big CM points. Now fix a small CM point $Z=(z_1, z_2) \in \H^2$ and
write
\begin{equation}
  \label{eq:zi}
z_i =\frac{b_i+\sqrt{D_i}}{2a_i},~ a := a_1 a_2 > 0,~
w_\dd := w_\dd(z_1, z_2),~ \hbox{ and }
\tilde{w}_\dd := w_\dd(z_1, \overline{z_2}).
\end{equation}
We also denote
\begin{equation}
  \label{eq:Ds}
 D_0= \hbox{gcd}(D_1, D_2) <0,~ D= \hbox{lcm}(D_1, D_2) <0,~ D_i= t_i^2 D_0,~ t:= t_1t_2.
\end{equation}
Then $D=t^2 D_0$ and $\gcd(t_1, t_2) = 1$ and
\begin{equation}
  \label{eq:samedisc}
  D_1 = D_2 \Leftrightarrow t = t_1 = t_2 = 1.
\end{equation}
Denote the associated negative two plane and its orthogonal complement in  $V_\R$ by
\begin{equation}
  \label{eq:UR}
 U_\R
 :=
  \R \Re(w_\dd) + \R \Im (w_\dd),~
 U^\perp_\R
 :=
  \R \Re(\tw_\dd) + \R \Im (\tw_\dd).
\end{equation}
Then $U_\R = U \otimes \R$ and $U_\R^\perp = U^\perp \otimes \R$ with
\begin{equation}
  \label{eq:Uab}
  U^\perp:= U_\R^\perp \cap V,~ \hbox{ and } ~   U:= U_\R \cap V.
\end{equation}
It is easy to check the following isometries of rational quadratic spaces
\begin{equation}
  \label{eq:ipminverse}
  \begin{split}
(U^\perp, Q_\dd) &\cong ( \kay_{D_0},  \frac{\dd t}{a} \norm), \quad
\tilde{\mu} \mapsto \frac{a}{t \sqrt{D_0}}(\tilde{\mu}, \overline{\tw_\dd})_\dd,\\
    (U, Q_\dd) &\cong ( \kay_{D_0},  -\frac{\dd t}{a} \norm), \quad
\mu \mapsto -\frac{a}{t \sqrt{D_0}}(\mu, \overline{w_\dd})_\dd.
  \end{split}
\end{equation}
We will denote the inverse maps by $i^+_\dd: \kay_{D_0} \to U^\perp$ and $i^-_\dd: \kay_{D_0} \to U$ respectively, which are given by
\begin{equation}
  \label{eq:idpm}
  i^+_\dd(\tilde{\lambda}) :=  -\frac{\dd}{\sqrt{D_0}} ( \tilde{\lambda} \tw_\dd - \overline{\tilde{\lambda}\tw_\dd}),~
i^-_\dd(\lambda) :=  -\frac{\dd}{\sqrt{D_0}} (\lambda w_\dd - \overline{\lambda w_\dd} ).
\end{equation}
Using the following properties
$$
-(w_\dd, \overline{w_\dd})_\dd = (\tw_\dd, \overline{\tw_\dd})_\dd
= - \frac{ tD_0}{\dd a},~
(w_\dd, \tw_\dd)_\dd =
(w_\dd, \overline{\tw_\dd})_\dd =
0,
$$
it is straightforward to verify that $i_\dd^\pm$ are the inverses to the two maps in \eqref{eq:ipminverse}.
Their sum then gives an isometry
\begin{equation}
  \label{eq:mapi}
  \begin{split}
      i_\dd: \kay_{D_0} \times \kay_{D_0} &\to V = U^\perp \oplus U\\
(\tilde{\lambda}, \lambda) &\mapsto i^+_\dd(\tilde{\lambda}) + i^-_\dd(\lambda),
  \end{split}
\end{equation}
whose inverse is given by
\begin{equation}
  \label{eq:iinverse}
  \begin{split}
      i_\dd^{-1}: V &\to \kay_{D_0} \times \kay_{D_0} \\
\mu_0 &\mapsto \lp
 \frac{a}{t \sqrt{D_0}}(\mu_0, \overline{\tw_\dd})_\dd,
- \frac{a}{t \sqrt{D_0}}(\mu_0, \overline{w_\dd})_\dd
\rp.
  \end{split}
\end{equation}

\begin{lemma}
  \label{lemma:equivact}
For any $(\tilde{\lambda}, \lambda) \in \kay_{D_0} \times \kay_{D_0}$ and $r \in \kay_{D_0}$, we have
\begin{equation}
  \label{eq:equivact}
  i_\dd(r\tilde{\lambda}, r\lambda) = \iota_{z_1}(r)
  i_\dd(\tilde{\lambda}, \lambda).
\end{equation}
\end{lemma}

\begin{proof}
  This follows from the definition and the identities
  \begin{equation}
    \label{eq:Twid}
    rw(z_1, z_2) =  \binom{r z_1}{r} (1~ -z_2) = \iota_{z_1}(r)\binom{ z_1}{1} (1~ -z_2) =
\iota_{z_1}(r) w(z_1, z_2)
  \end{equation}
and $r\tw = \iota_{z_1}(r) \tw$.
\end{proof}

Now we have again  $\GSpin(U) =T =\kay_D^\times$, with the diagram:
\begin{equation} \label{eq:torus2}
\xymatrix{
 1 \ar[r]  &\G_m \ar[r] \ar[d] &T=\kay_D^\times \ar[r] \ar[d] &\SO(U)=\kay_D^1\ar[r] \ar[d] &1
\cr
 1 \ar[r]  &\G_m \ar[r] &H\ar[r] &\SO(V) \ar[r] &1.
 \cr
}
\end{equation}

\begin{lemma}\label{lem:Embedding2} The map $T\rightarrow H$ in diagram (\ref{eq:torus2}) is induced by
$$
\iota_Z =(\iota_{z_1}, \iota_{z_2}):  \kay_D \rightarrow M_2(\Q) \times M_2(\Q).
 $$
 Here $\iota_{z_i}$ is the map defined in Lemma \ref{lem:Embedding1}.
\end{lemma}
\begin{proof}First notice that
$$
\tilde\gamma= \tilde\alpha (\tilde\beta)^{-1} =\iota_{z_1}(\sqrt{D_0})
$$
is independent of $z_2$, and that
$$
\tilde\delta = (\tilde\beta)^{-1} \tilde\alpha = \iota_{z_2}(\sqrt{D_0})
$$
is independent of $z_1$. Now
\begin{align*}
\Gspin(U) &= \{g=(g_1, g_2) \in H:  g \hbox{ acts on  } U^\perp \hbox{ trivially}\}
\\
 &= \{ (g_1,g_2) \in H:\, g_1 \tilde\gamma =\tilde\gamma g_1,  g_2 \tilde\delta =\tilde\delta g_2,  g_1 \tilde\beta = \tilde\beta  g_2\}
 \\
 &= \{ (g_1, g_2) \in  \GL_2 \times \GL_2:\,  g_2 =\iota_{z_2}(r),  g_1 =\iota_{z_1}(r),  \hbox{ for some } r \in  \kay_D^\times\}.
\end{align*}
Here the last equality comes from the identities
$$
(a_2, \frac{-b_2+\sqrt{D_2}}2) =(a_1, \frac{-b_1+\sqrt{D_1}}2) \frac{2a_2}{t_1} \tilde\beta,
\quad \hbox{ and }  \iota_{z_1} (r) = \tilde\beta \iota_{z_2}(r) \tilde\beta^{-1}.
$$
So by Lemma~\ref{lem:Embedding1} we have the desired map
$$
\iota_Z:   T \rightarrow H, \quad \iota_Z(r) = (\iota_{z_1}(r), \iota_{z_2}(r)).
$$
\end{proof}


\subsection{Lattices}
 \label{subsec:lattice}
To understand the lattices,
 it is better to have some restrictions on $z_i$, which are harmless in this paper because of the following lemma, whose proof is left to the reader.

\begin{lemma} \label{lem:smallCM} Suppose that $D_i$ satisfy \eqref{eqd}. For ideal classes $\Ac_i$ of $\OO_{D_i}$, $i=1, 2$, there are ideals $\mathfrak a_i=[a_i, t_i \frac{b+\sqrt{D_0}}2]$ representing $\Ac_i$ with associated CM points $z_i= t_i\frac{b+\sqrt{D_0}}{2  a_i}$  such that $a_i >0$,
  \begin{equation}
    \label{eq:abccond}
    \begin{split}
      \gcd(a_1, a_2)&=1,~ \gcd(a_{i}, 6b D) =1,~ b^2 \equiv D_0 \bmod (4 a_{i}),~
      48 \mid \gcd(a_i - t_i, b - 1).
    \end{split}
  \end{equation}
\end{lemma}
\begin{remark}
  \label{rmk:zicond}
For any fixed $n \in \N$ and odd $r_i \in \Z$, we can find $a_i$ such that
$2^n \mid a_i - r_i$ for $i = 1, 2$.
\end{remark}

\begin{lemma}
  \label{lemma:ziinv}
  Let $D_j$ be discriminants satisfying \eqref{eqd} and $z_i$ CM points chosen as in Lemma \ref{lem:smallCM}. Then
  \begin{equation}
    \label{eq:ziinv}
    \frac{\ff_2(z_1)}{\ff_2(z_2)} =     \frac{f(\af_1)}{f(\af_2)}.
  \end{equation}
  with $f(\af_i)$ the class invariant defined in \eqref{eq:classinv}.
\end{lemma}

\begin{proof}
  This is equivalent to checking that
  $3(D_i - 1)+ b_i (a_i - c_i + a_i^2 c_i) \bmod{48}$ is independent of $i = 1, 2$, where $b_i := bt_i, c_i := \frac{t_i^2}{a_i} c, c:= \frac{b^2 - D_0}{4}$. By the choice of $z_i$, we have
  $$
a_i \equiv t_i \bmod{48},~ c_i \equiv t_i c \bmod{48}.
$$
Furthermore, $2 \mid D_0 + b$ and $2 \mid c$.
Substituting these in and using $c \in 2\Z,~ 2t_i^2 \equiv 2 \bmod{48}$ gives us
\begin{align*}
  3(D_i - 1) &+ b_i (a_i - c_i + a_i^2 c_i)
  \equiv 3(D_0 t_i^2 - 1) +  bt_i^2 (1 - c + t_i^2 c) \\
  & \equiv \frac{3D_0 + b}{2} 2t_i^2 - 3 -b \frac{c}{2} t_i^2 (2 - 2t_i^2) \equiv 3D_0 + b - 3   \bmod{48},
\end{align*}
which is independent of $i = 1,2$.
\end{proof}
  From now on, we will suppose that the $z_i$ satisfy  the conditions in  Lemma \ref{lem:smallCM}.
Let
$$
a:=a_1a_2, \quad c := \frac{b^2 - D_0}{4a} \in 16\Z.
$$
The following integral ideals will be important for us
\begin{equation}
  \label{eq:af}
  \begin{split}
    \af_i &:= a_i (\Z + \Z z_i) \subset \OO_{D_i},~
    \bfrak_i := a_i (\Z + 2\Z \overline{z_i}) = \overline{\af_i} \cap \OO_{4D_i} \subset \OO_{4D_i}, \\
    \mathfrak a_0 &= \Z a + \Z\frac{-b+\sqrt{D_0}}2 = \overline{\af_1}\overline{\af_2} \subset \OO_{D_0},\\
    \mathfrak a&= \Z a + \Z( -b +\sqrt{D_0} ) = \af_0 \cap \OO_{4D_0} \subset \OO_{4D_0}.
  \end{split}
\end{equation}
In addition, we choose $\tilde{b} \in \Z$ such that
\begin{equation}
  \label{eq:tb}
  \begin{split}
      \tilde{b} + (-1)^i b \equiv 0 &\bmod 2a_i,~ i = 1, 2, \\
b \equiv \tb &\bmod 4D.
  \end{split}
\end{equation}
and define the following ideals 
\begin{equation}
  \label{eq:tildea}
  \begin{split}
      \tilde{\af}_0 &:= \Z a+ \Z \frac{-\tilde{b} + \sqrt{D_0}}{2} = \overline{\af_1}\af_2 \subset \OO_{D_0},\\
    \tilde{\af} &:= \Z a +\Z( -\tilde{b} + \sqrt{D_0}) = \taf_0 \cap \OO_{4D_0} \subset \OO_{4D_0}.
  \end{split}
\end{equation}
Notice that $\mathfrak a_0$ and $\taf_0$ are integral ideals of $\OO_{D_0}$, while $\mathfrak a$ and $\taf$ are integral ideals of  $\OO_{4D_0}$.
Furthermore,
\begin{equation}
  \label{eq:norma}
a = [\OO_{D_0}: \af_0]
= [\OO_{D_0}: \taf_0]
= [\OO_{4D_0}: \af]
= [\OO_{4D_0}: \taf].
\end{equation}
Using $\tilde\af, \af \subset \Zb[\sqrt{D_0}]$, we can make the following canonical identification
\begin{equation}
    \label{eq:Aid}
    \tilde\af'/\tilde\af = \frac1{2\dd \sqrt{D}} \tilde\af/\tilde\af
    \cong
    A
    \cong     \frac1{2 \dd \sqrt{D}} \af/\af = \af'/\af,~
    A :=     \frac1{2\dd \sqrt{D}} \Zb[\sqrt{D_0}] / \Zb[\sqrt{D_0}].
\end{equation}
Also, denote $A_p := A \otimes \Zb_p$.

Now, we can view $t_{}\af_i/t_i$ and $t\af, t\taf$ as invertible $\OO_D$ ideals.
Let $\Ac_{i} \in \Cl(D_i)$ denote their classes.
The following lemma will be helpful for us later.
\begin{lemma}
  \label{lemma:const}
  In the notation above,
  there exist $\tilde{\alpha} \in \taf$ such that $\Nm(\tilde{\alpha}) = a/t$ if and only if the classes $\Ac_i$ are the same, i.e.\ $D_1 = D_2 = D$ and $\Ac_1 = \Ac_2 \in \Cl(D)$.
  If this happens, then $t = 1$.
\end{lemma}

\begin{proof}
  If $t = 1$, then $\Ac_1 = \Ac_2$ if and only if $\taf_0 = \overline{\af_1} \af_2 = \tilde{\alpha} \OO_{D_0}$.
  As $D_0 \equiv 1 \bmod{8}$ and $\Nm(\taf_0) = a \equiv 1 \bmod{16}$, we must have $\tilde{\alpha} \in \OO_{4D_0} \cap \taf_0 = \taf$.

  If $t > 1$, then $D_1 \neq D_2$ and $\Ac_1\neq \Ac_2$.
  On the other hand, we have
  $$
a = \Nm(\taf) = [\OO_{4D_0}: \taf] \le [\OO_{4D_0}: \tilde{\alpha} \OO_{4D_0}] = \Nm(\tilde{\alpha})
$$
for any $\tilde{\alpha} \in \taf$, so there cannot be any $\tilde{\alpha} \in \taf$ with norm $a/t < a$.
\end{proof}
We can now embed the ideals $\tilde{\af}_0$ and $\tilde{\af}$ into the lattice $M_2(\Z)$ as follows.

\begin{proposition} \label{prop:Lattice} \quad Let $L_0 =M_2(\Z) \subset V$.
For $\dd>0$ a positive integer, denote
\begin{equation}
  \label{eq:Ld}
  L = L_\dd := \left\{\lambda \in L_0:
\lambda \equiv \pmat{*}{*}{0}{*} \bmod{2}
\right\}
\end{equation}
with the quadratic form $Q_\dd := \dd \det$.
Let $\mathcal N = \Nc_\dd :=L_\dd\cap U$ and $\mathcal P = \Pc_\dd :=L_\dd \cap U^\perp$.  Then
\begin{equation}
  \label{eq:latticeimage}
  \begin{split}
i_\dd^+ (\tilde{\mathfrak a}_0 ) &=L_0\cap U^\perp,  \quad i_\dd^+ (\tilde{\mathfrak a}) =\mathcal P, \\
i_\dd^-(\mathfrak a_0 ) &=L_0\cap U,  \quad  i_\dd^- (\mathfrak a) =\mathcal N.
  \end{split}
\end{equation}
The dual lattices of $\tilde{\af}$ and $\af$ are $\frac{1}{2 \dd t \sqrt{D_0}} \tilde{\af}$ and $\frac{1}{2 \dd t \sqrt{D_0}} {\af}$ respectively.
Furthermore, we have
\begin{equation}
  \label{eq:latcond}
  \begin{split}
    i_\dd(\tilde{\lambda}, \lambda) \in L_\dd' &\Leftrightarrow \pmat{1}{1}{\overline{z_2}}{z_2} \binom{\tilde{\lambda}}{\lambda} \in \frac{1}{\dd t_1} \overline{\af_1} \times \frac{1}{2 \dd t_1} \bfrak_1,\\
    i_\dd(\tilde{\lambda}, \lambda) \in L_\dd &\Leftrightarrow \pmat{1}{1}{\overline{z_2}}{z_2} \binom{\tilde{\lambda}}{\lambda} \in \frac{1}{t_1} \bfrak_1 \times \frac{1}{ t_1} \overline{\af_1}.
  \end{split}
\end{equation}
This implies that for any $(\tilde{\lambda}, \lambda) \in \frac{1}{2 \dd t \sqrt{D_0}} \tilde{\af} \times \frac{1}{2 \dd t \sqrt{D_0}} \af$ and $\mu_0 \in L_\dd'$,
$$
i_\dd(\tilde{\lambda}, \lambda) \equiv \mu_0 \bmod L_{\dd} \Leftrightarrow
\pmat{1}{1}{\overline{z_2}}{z_2} \binom{\tilde{\lambda} - \frac{a}{t \sqrt{D_0}} (\mu_0, \overline{\tw_\dd})_\dd}{\lambda + \frac{a}{t \sqrt{D_0}}  ( \mu_0, \overline{w_\dd})_\dd} \in \frac{1}{ t_1} \bfrak_1 \times \frac{1}{ t_1} \overline{\af_1}.
$$
\end{proposition}
\begin{remark}
Since
$$\mathcal P \oplus \mathcal N \subset L \subset L' \subset \mathcal P' \oplus \mathcal N',
$$
we have $L' /(\mathcal P + \mathcal N)  \subset \mathcal P'/\mathcal P \oplus \mathcal N'/\mathcal N$ and the natural  projection $L' /(\mathcal P + \mathcal N) \rightarrow L' /L $.
\end{remark}
\begin{proof}
To characterize the preimage of $L_\dd \cap U$ under $i^-_\dd$, we can use the fact that $L_\dd$ and $L'_\dd$ are dual lattices, and apply the discussion in section \ref{subsec:smallCM} to see that
\begin{align*}
  i^-_\dd(\lambda) &= \frac{1}{\sqrt{D_0}}(\lambda w - \overline{\lambda w}) \in L_\dd \Leftrightarrow (  i^-_\dd(\lambda) , \mu)_\dd \in \Z \text{ for all }\mu
\in L'_\dd \\
& \Leftrightarrow \frac{1}{\sqrt{D_0}} \lp \lambda w - \overline{\lambda w}, \frac{1}{2\dd} \pmat{\mu_1}{\mu_2}{\mu_3}{\mu_4} \rp_\dd \in \Z \text{ for all } \mu_1, \mu_3, \mu_4 \in 2\Z, \mu_2 \in \Z\\
& \Leftrightarrow  \tr \lp  \frac{ a\lambda (-\mu_1 z_2 + \mu_4 z_1 - \mu_2 + \mu_3 z_1 z_2) }{2a\sqrt{D_0}} \rp \in \Z \text{ for all } \mu_1, \mu_3, \mu_4 \in 2\Z, \mu_2 \in \Z\\
& \Leftrightarrow
 \tr \lp  \lambda \overline{\alpha} \rp \in \Z \text{ for all } \alpha  \in \frac{1}{2a \sqrt{D_0}} \af
 \Leftrightarrow \lambda \in \af.
\end{align*}
The same argument shows the other three equations in \eqref{eq:latticeimage}.

For \eqref{eq:latcond}, we argue similarly, and notice that for $\mu = \smat{\mu_1}{\mu_2}{\mu_3}{\mu_4} \in V$,
\begin{align*}
  (i_\dd(\tilde{\lambda}, \lambda), \mu)_\dd
  &=
\dd
\tr
\lp
 \frac{(\tilde{\lambda} + \lambda) (\mu_4 z_1 - \mu_2) + (\tilde{\lambda} \overline{z_2} + \lambda z_2)(\mu_3 z_1 - \mu_1) }{\sqrt{D_0}}
    \rp\\
  &=
    \tr
\lp
(\tilde{\lambda} + \lambda) \gamma
\rp +
\tr
\lp
(\tilde{\lambda} \overline{z_2} + \lambda z_2) \delta \rp,
\end{align*}
with $\gamma = \dd \frac{\mu_4 z_1 - \mu_2}{\sqrt{D_0}}, \delta = \dd \frac{\mu_3 z_1 - \mu_1}{\sqrt{D_0}}$.
So for $\mu \in L_\dd$, we have $\mu_1, \mu_2, \mu_4 \in \Z$ and $\mu_3 \in 2\Z$. The condition $i_\dd(\tilde{\lambda}, \lambda) \in L_\dd'$ is then equivalent to
$$
\tr
\lp
(\tilde{\lambda} + \lambda) \gamma
\rp \in \Z,~
\tr
\lp
(\tilde{\lambda} \overline{z_2} + \lambda z_2) \delta \rp \in \Z
$$
for all $\gamma \in \frac{\dd}{\sqrt{D_0}}\Z[z_1]$
and $\delta \in \frac{\dd}{\sqrt{D_0}}\Z[2z_1]$,
which is equivalent to $\tilde{\lambda} + \lambda \in \frac{1}{\dd t_1} \overline{\af_1}$ and
$\tilde{\lambda} \overline{z_2} + \lambda z_2 \in \frac{1}{2\dd t_1} \bfrak_1$.
Similarly, $\mu \in L_\dd'$ is equivalent to $\mu_1, \mu_3,  \mu_4 \in \frac{1}{\dd} \Z$ and $\mu_2 \in \frac{1}{2\dd} \Z$, which means $\gamma \in \frac{1}{2\sqrt{D_0}}\Z[2z_1]$ and $\delta \in \frac{1}{\sqrt{D_0}} \Z[z_1]$.
Therefore, $i_\dd(\tilde{\lambda}, \lambda) \in L_\dd$ if and only if $\tilde{\lambda} + \lambda \in \frac{1}{t_1}\bfrak_1$ and $\tilde{\lambda} \overline{z_2} + \lambda z_2 \in \frac{2}{t_1} \overline{\af_1}$.
The last claim now follows from \eqref{eq:iinverse}.
\end{proof}

\subsection{Level} Now  we assume $\dd|24$. Let $\Gamma^\dd$ be the subgroup of $\Gamma_0(2)$ generated by $\Gamma_0(2)^{\mathrm{der}}$, $T^\dd$, $S^2$, and $TB$, where
$$
T =\kzxz {1} {1}{0} {1},  \quad S =\kzxz {0} {-1} {1} {0} ,~ \text{ and }  B= \kzxz {1} {0} {-2} {1}.
$$
This was called $\Gamma_{\chi, \dd}$ in \cite{LY}, since $\Gamma^\dd = \ker(\chi^{24/\dd})$ with $\chi: \Gamma_0(2) \to \C^\times$ is the character defined by
\begin{equation}
  \label{eq:chi}
\ff_2(\gamma z) = \chi(\gamma) \ff_2(z),~ \gamma \in \Gamma_0(2).
\end{equation}
In particular,
\begin{equation}
  \label{eq:G24}
\Gamma^1=\Gamma_0(2), \hbox{ and }   \Gamma^{24} = \{\gamma \in \Gamma_0(2): \ff_2(\gamma z) = \ff_2(z) \text{ for all }z \in \H\}.
\end{equation}
$\Gamma^\dd$ is a normal subgroup of $\Gamma_0(2)$ of index $\dd$.
It is easy to check that $TB \in \Gamma^{24}$,since
\begin{equation}
  \label{eq:chiTB}
  \chi(T) = \chi(B)^{-1} = \zeta_{24}.
\end{equation}
Since the eta product $\ff_2(z)^{24/\dd}$ has level $\Gamma(2\dd)$ (e.g.\ see Theorem 1.7 and section 2.1 of \cite{Koehler}), $\Gamma^\dd$ is a congruence subgroup containing $\Gamma(2\dd)$, and $\chi$ is a character of $\SL_2(\Z/48\Z) \cap L_{48}$---the image of $\Gamma_0(2)$ in $\SL_2(\Z/48\Z)$, 
where
$$
L_{48} = L \otimes \Z/48\Z = \{\gamma \in M_2(\Z/48\Z): \gamma \equiv \smat{*}{*}{0}{*} \bmod{2} \}.$$
The values of $\chi$ can be computed using eta multipliers, and expressed in the following way.
\begin{proposition}
  \label{prop:chival}
  Let $\chi = \chi_2 \chi_3$ with $\chi_p$ the $p$-component of the character defined by $\chi_2 := \chi^9, \chi_3 := \chi^{16}$.
  Then $\chi(\gamma) = \chi_2(\gamma_2) \chi_3(\gamma_3)$ for any $\gamma = (\gamma_2, \gamma_3) \in \SL_2(\Z/48\Z) \cap L_{48}$.
Furthermore, $\chi_p$ is given by
  \begin{equation}
    \label{eq:chip}
    \begin{split}
      \chi_2\lp \pmat{a}{b}{c}{d}_2 \rp &=
      \lp \frac{2}{a} \rp \zeta_8^{3a(b + c/2)},~
      \chi_3\lp \pmat{a}{b}{c}{d}_3 \rp =
 \zeta_3^{-(a + d)c + bd(c^2 - 1)}
    \end{split}
  \end{equation}
  for any $\gamma = \smat{a}{b}{c}{d} \in \SL_2(\Z/48\Z) \cap L_{48}$.
\end{proposition}
\begin{remark}
  \label{rmk:chi3}
Note that $c^2 \equiv 0, 1 \bmod{3}$ depending on whether 3 divides $c$ or not.
\end{remark}
\begin{proof}
  Since $\chi_p$ only depends on $\gamma$ modulo powers of $p$, the first claim clearly holds.
  For the explicit formula of $\chi_2$, we know that $\Gamma(48) \subset \ker(\chi)$ and $\Gamma(16) \subset \ker(\chi_2)$. So given $\gamma_2 = \smat{a}{b}{c}{d} \in \SL_2(\Z/16\Z)$ with $c \in 2\Z/16\Z$, we can apply Lemma 6 in \cite{Gee} to write
  $$
  \gamma_2 \equiv ST^{-(1 + c)a^{-1}}ST^{-a}ST^{b(1+c)a^{-1}-d}
\equiv B^{(1-(1 + c)a^{-1})/2} T B^{(1 -a)/2} T^{b(1 + c)a^{-1} - d}
  \bmod{16}
  $$
 with $a^{-1}$ the inverse of $a \bmod{16}$.
  Therefore, we have
  \begin{align*}
    \chi_2(\gamma)
    &=
      \chi(B^{(1-(1 + c)a^{-1})/2} T B^{(1 -a)/2} T^{b(1 + c)a^{-1} - d})^9 \\
    &= \zeta_8^{3(((1 + c)a^{-1}-1)/2 + 1 + (a-1)/2 + b(1 + c)a^{-1} - d)}.
  \end{align*}
  Using $a \equiv a^{-1} \bmod{8}$ and $(a + a^{-1})/2 \equiv a + (a^2 - 1)/2 \bmod{8}$, it is easy to check that
 \begin{align*}
   ((1 + c)a^{-1} - 1)/2  + 1 + (a-1)/2 &\equiv (a + a^{-1})/2 + a^{-1}c/2 \\
   &\equiv a(1 + c/2) + (a^2 - 1)/2 \bmod{8},\\
   b(1 + c)a^{-1} - d &\equiv ab(1 + c) - d \\
   &\equiv ab + a^2d - a - d \equiv a(b - 1) \bmod{8}.
  \end{align*}
  Substituting these in and applying $\lp \frac{2}{a} \rp = (-1)^{(a^2 - 1)/8}$ gives us the formula \eqref{eq:chip}.
    The case of $\chi_3$ is similar and we leave its proof to the reader.
\end{proof}
Let $K_\dd^0$ be the subgroup of $\GL_2(\hat\Z)$ generated by $\nu(\hat\Z^\times)$, and the preimage of $\Gamma^\dd/\Gamma(2\dd)$.
Then $K_\dd^0$ is also invariant with respect to conjugation by the preimage of $\Gamma_0(2)$ in $\SL_2(\hat\Z) \subset \GL_2(\hat\Z)$,
under the projection $\GL_2(\hat\Z) \rightarrow \GL_2(\Z/{2\dd\Z})$. Let
$$
K_\dd =H(\A_f) \cap (K_\dd^0 \times K_\dd^0),~ K_\dd' =\langle K_\dd, (T, T)\rangle.
$$
Notice that $K'_1 = K_1$ .
 Then we have as Shimura varieties
\begin{align*}
X_\dd^0 &= \GL_2(\Q) \backslash \mathbb D^0 \times \GL_2(\A_f)/K_\dd^0 =\Gamma^\dd \backslash \H,
\\
X_\dd&= H(\Q) \backslash \mathbb D \times H(\A_f)/K_\dd   = (\Gamma^\dd \backslash \H)^2,
\end{align*}
and a natural projection
$$
X_\dd \rightarrow   X_\dd' =  H(\Q) \backslash \mathbb D \times H(\A_f)/K_\dd' =H_\dd'\backslash \H^2.
$$
Here $H_\dd':= K_\dd' \cap  H(\Q)^+ = \langle \Gamma^\dd \times \Gamma^\dd,  (T, T)\rangle$.
We use $X_\dd^\Delta \subset X_\dd$ to denote the divisor given by the image of the diagonal embedding $X_\dd^0 \hookrightarrow X_\dd$, and
\begin{equation}
  \label{eq:Xdj}
X^\Delta_\dd(j) := (T^j \times \mathds{1})^*(X^\Delta_\dd) \subset X_\dd
\end{equation}
the pullback along the translation by $(T, 1)$ map (see (4.3) of \cite{LY}).
Note that $X^\Delta_\dd(j)$ descends to a divisor on $X'_\dd$ since $\Gamma^\dd \subset \Gamma_0(2)$ is a normal subgroup.
We slightly abuse notation and use $X^\Delta_\dd(j)$ to denote this divisor on $X_\dd'$.

The function $|\mathfrak f_2(z_1) -\epsilon \mathfrak f_2(z_2)|$  ($\epsilon =\pm 1$) is a (non-holomorphic) modular function on $ X_\dd'$.
First, we have the following generalization of \cite[Lemma 5.2]{LY}.

\begin{lemma}
\label{lemma:KdT}
Let $D_1, D_2$ be discriminants satisfying \eqref{eqd}, and $Z= (z_1, z_2) \in \H^2$ a small CM point with $z_i = \frac{b_i + \sqrt{D_i}}{2a_i}$ having discriminant $D_i$ and satisfying
\begin{equation}
  \label{eq:zicond1}
\gcd(6, a_1a_2) = 1,~ a_1 b_2 \equiv a_2 b_1 \bmod{48}.
\end{equation}
Denote$$
\iota_Z= (\iota_{z_1}, \iota_{z_2}):  T(\A_f)=\kay_{D,f}^\times  \rightarrow H(\A_f)
$$
the associated embedding from the torus $T$ and $H$.
For $\dd|24$, let $K_{\dd, T} \subset T(\A_f)$ be the preimage of $K_\dd'$. Then
$$
K_{\dd, T} = \hat{\OO}_D^\times
$$
is independent of $\dd$, where $D = \mathrm{lcm}(D_1, D_2)$.
\end{lemma}
\begin{remark}
  \label{rmk:Kd0}
  If $z_0 = \frac{b + \sqrt{D}}{2a}$ is a CM point with discriminant $D \equiv 1 \bmod{8}$ and $2 \nmid a$, then $\iota_{z_0}^{-1}(K_1^0) = \hat{\OO}_D^\times$ by Lemma 5.4 in \cite{YY19}.
  However, $\iota_{z_0}^{-1}(K_\dd^0)$ will be strictly contained in $\hat{\OO}_D^\times$ if $\dd > 1$.
\end{remark}

\begin{remark}
  \label{rmk:zicond1}
  If $z_i$ are chosen as in Lemma \ref{lem:smallCM}, then they also satisfy \eqref{eq:zicond1}.
\end{remark}

\begin{remark}
This gives another proof of Lemma 5.2 in \cite{LY} without resorting to explicit calculations. 
\end{remark}
\begin{proof}
First of all, we have $\iota(\hat{\OO}_D^\times) \subset H(\hat\Z)$. 
Since $K_1' = K_1 = H(\A_f) \cap (K_1^0 \times K_1^0)$, Remark \ref{rmk:Kd0} implies that
$$
K_{1, T} = \iota^{-1}(K_1') = \iota^{-1}(K_1) =
\iota^{-1} (K_1^0 \times K_1^0) = \iota_{z_1}^{-1}(K_1^0) \cap \iota_{z_2}^{-1}(K_1^0) = \hat{\OO}_{D_1}^\times
\cap \hat{\OO}_{D_2}^\times
= \hat{\OO}_{D}^\times.
$$
Since $K_{\dd, T} \subset K_{1, T}$ by definition, we just need to show that $ \hat{\OO}_{D}^\times \subset K_{\dd, T}$, or equivalently $\iota(\hat{\OO}_{D}^\times) \subset K_{\dd}'$.
Furthermore, we only need to check the places above 2 and 3, as $K_{\dd, T}$ and $K_{1, T}$ are the same everywhere else.
Since the map $\iota_{z_i}$ then only depends on $a_i, b_i, D_i$ modulo 48, we can apply condition \eqref{eq:zicond1} and \eqref{eq:iotaz2} to check that for all $r = \alpha + \beta\sqrt{D} \in (\OO_{D}/48\OO_D)^\times \cong (\OO_{D_0}/48 \OO_{D_0})^\times$
$$
\iota_{z_i}(r) \equiv \iota_{z_i'}(r'_i) \bmod{48}
$$
where $r'_i = \alpha + a_i \beta\sqrt{D_i'} \in (\OO_{D}/48 \OO_{D})^\times$ and $z_i' := \frac{b' + \sqrt{D_i'}}{2}$ with $b' \in \Z$ and $D_i' < 0$ a discriminant satisfying
$$
b' \equiv b_i a_i^{-1} \bmod{48},~ D_i' \equiv a_i^{-2} D \bmod{48}.
$$
Note that $D_i'$ still satisfies \eqref{eqd} as $\gcd(a_i, 6) = 1$ and
$$
d := r_i' \overline{r_i'} \bmod{48}
$$
is independent of $i = 1, 2$.
Now by Prop.\ 12 and 13 in \cite{Gee} and Remark \ref{rmk:iotaz}, we know that $\zeta_{48}^{b'-4}\ff_2$ is invariant under the action of
$$
 W_{48, z_i'} := \iota_{z_i'}((\OO_D/48\OO_D)^\times) \subset \GL_2(\Z/48\Z),
$$
given in (1) of \cite{Gee}, i,e.\
there exists $\gamma_i \in \Gamma^{24}/\Gamma(48)$ such that
$$
\iota_{z_i'}(r'_i)  \equiv \nu(d) \gamma_i T^{\frac{1 - d}{2} (b' - 4)} \bmod{48},
$$
which implies that
$$
\iota(r) = (\nu(d), \nu(d)) (\gamma_1, \gamma_2) (T, T)^{(1-d)(b'-4)/2} \in K_{24}'.$$
This finishes the proof.
\end{proof}

\section{Borcherds products and small CM value formula}

\subsection{A brief review of a result of \cite{LY}} The first step to prove Theorem~\ref{thm:YZconj}  is a result of \cite[Theorem 1.8]{LY} to write $\mathfrak f_2(z_1) - \mathfrak f_2(z_2)$
as product of Borcherds products, which we now review.

Let $\dd|24$.
Endow  $V=V_\dd=M_2(\Q)$  with quadratic form  $Q=Q_\dd = \dd \det$. The lattice
$$
L=L_\dd=\{ \kzxz {a} {b} {c} {d}  \in M_2(\Z) :\, 2|c\} \subset V
$$
is even integral with respect to $Q_\dd$.
Denote $\omega_{\dd}$ and $\omega_{\dd, p}$ the Weil representation of $\SL_2(\Z)$ on $S(L_{\dd})$ and $S(L_{\dd} \otimes \Z_p)$ respectively.
Define
\begin{equation}
  \label{eq:ufd}
  \begin{split}
    \phi_\dd &:=
    \sum_{\substack{s \in (\Z/\dd\Z)^\times\\ \gamma \in L/\dd L \\ \det(\gamma) \equiv 1 \bmod{\dd}}}
\chi(\gamma)^{(24/\dd)s} \phi_{\frac{1}{\dd} \gamma}
=    \sum_{j \in \Zb/\dd\Zb} a_\dd(j) \lp \sum_{\mu \in \frac{1}{\dd}T^j \Gamma^{ \dd}/L} \phi_\mu \rp \in S(L_\dd),\\
a_\dd(j) &:= \lp \sum_{s \in (\Zb/\dd\Zb)^\times} \zeta_{\dd}^{sj} \rp = \mu\lp \frac{\dd}{(\dd, j)} \rp \frac{\varphi(\dd)}{\varphi(\dd/(\dd, j))} \in \Zb.
  \end{split}
\end{equation}
  Note that the coefficient $a_{\dd_p}(j)$ is non-zero in exactly the following cases.
\begin{table}[h]
\begin{tabular}{|c||c|c|c|c|c|c|c|}
 \hline
$(\dd_p, j)$    & $(\dd_p, 0)$ & $(2, 1)$ & $(4, 2)$ & $(8, 4)$ & $(3, 1)$ & $(3, 2)$ \\
 \hline
 $a_{\dd_p}(j)$ & $\varphi(\dd_p)$ & -1& -2 & -4 & -1&-1 \\
 \hline
 \end{tabular}
\caption{Nonzero values of $a_{\dd_p}(j)$.}
\end{table}
Using the vector $\phi_\dd$ and its translates, we were able to construct Borcherds products with suitable divisors in \cite{LY}.
\begin{theorem} (\cite[Theorem 4.4, 4.5]{LY}) \label{theo:BorcherdsLifting}
  For every $\dd \mid 24$, there exists $\tilde{F}_\dd \in M^!(\omega_d)^{H_\dd'}$ such that
  \footnote{The factor defining $\tilde F_\dd$ in Theorem 4.4 in \cite{LY} should be $\sqrt{2}^{24/\dd}$ instead of $\sqrt{2}^\dd$.}
  $$
  \tilde{F}_\dd (\tau) =
  q^{-1/\dd} \phi_{\dd} +
\delta_{\dd = 1} 24 \phi_{L + \smat{0}{1/2}00}
+ O(q^{1/(2\dd)}),
  $$
  and the divisor of its  Borcherds lifting $\Psi_\dd(z) = \Psi(z_1, z_2,\tilde F_\dd)$ (see e.g.\  \cite[section 4.1]{LY})   is given by


    $$
    \mathrm{Div } \Psi_\dd(z) = \sum_{j \in \Zb/\dd\Zb} a_\dd(j) X_\dd^\Delta(j),
    $$
where $X^\Delta_\dd(j) \subset X'_\dd$ is defined in \eqref{eq:Xdj}.
Finally, we have 
\begin{equation}
  \label{eq:Blift}
  \lp \ff_2(z_1)^{24/s} - \ff_2(z_2)^{24/s} \rp^{s} = \prod_{\dd \mid s}
\Psi_\dd(z_1, z_2)
\end{equation}
for every $s|24$.
\end{theorem}


\begin{remark}
  \label{rmk:Psid}
  We can invert equation \eqref{eq:Blift} and express
  $$
  \Psi_\dd(z_1, z_2) =  \prod_{s \mid \dd}   \lp \ff_2(z_1)^{24/s} - \ff_2(z_2)^{24/s} \rp^{s \cdot \mu(\dd/s)}.
$$
\end{remark}

  It is easy to check that the coefficients $a_\dd(j)$ are multiplicative, i.e.
\begin{equation}
  \label{eq:adjmult}
  a_\dd(j) = \prod_{p \text{ prime}} a_{\dd_p}(j \bmod \dd_p).
\end{equation}
Therefore, we have $\phi_\dd = \otimes_{p < \infty}\phi_{\dd, p}$ with $\phi_{\dd, p} \in S(L_{\dd} \otimes \Z_p)$ given by
\begin{equation} \label{eq:u}
\phi_{\dd, p} :=
  \sum_{j \in \Zb/\dd_p\Zb}
a_{\dd_p}(j)
    \lp \sum_{\mu \in (\frac{1}{\dd}T^j \Gamma^{ \dd}/L_{}) \otimes \Z_p} \phi_\mu \rp,~ \text{ when }p \mid 6,
  \end{equation}
  and the characteristic function of $L_\dd \otimes \Zb_p = M_2(\Z_p)$ when $p \nmid 6$.
  So as a function on $\frac{1}{\dd} L \otimes \Zb_p$, we can write
  \begin{equation}
    \label{eq:phid23}
    \begin{split}
          \phi_{\dd} \lp \frac{\mu}{\dd} \rp &=
    \phi_{\dd, 2} \lp \frac{\mu}{\dd} \rp
    \phi_{\dd, 3} \lp \frac{\mu}{\dd} \rp  \prod_{p \nmid 6} \phi_{\dd, p}(\mu),\\
    \phi_{\dd, 2} \lp \frac{\mu}{\dd} \rp &=     \phi_{\dd_2, 2} \lp \frac{\mu}{\dd_2} \rp,~
    \phi_{\dd, 3} \lp \frac{\mu}{\dd} \rp =         \phi_{\dd_3, 3} \lp \frac{\mu}{\dd_3} \rp .
    \end{split}
  \end{equation}
  Furthermore, one can use $\dd_2^2 \equiv 1 \bmod{3}$ and $\dd_3^2 \equiv 1 \bmod 8$ to check that this identification intertwines $\omega_{\dd, 3}$ and $\omega_{\dd_3, 3}$.


  \subsection{Incoherent Eisenstein Series}
  \label{subsec:incoh}
Write $\kay =\Q(\sqrt D)$ and let $\epsilon =\epsilon_{\kay/\Q}$ the Dirichlet character associated to the quadratic field $\kay$. 
Recall that $\mathcal  N=(\mathcal N_\dd, Q_\dd) \cong (\mathfrak a, -\frac{\dd t}{a} \Nm)$ with $\dd \mid 24$. Write $\hat{\mathfrak a} =  \hat{\delta}  \hat{\OO}_{4D_0}$ with  $\hat{\delta} \overline{\hat{\delta}}=a u$ for some $u \in \hat{\OO}_{4D}^\times$. Then
$(\hat{\mathcal N}, Q_\dd) \cong (\hat{\OO}_{4D_0},  \kappa \Nm)$ with $\kappa =-\dd t u \in  \hat{\Z}$. Let $dx$ be the Haar measure on $\kay_{D, \A_f}$ such that $\hbox{Vol}(\hat{\OO}_D, dx) =|D|_{\A_f}^{\frac{1}2}=|D|^{-1/2}$.  This Haar measure depends only on the quiadratic field $\kay$ (not the choice of $D$).
Notice that $\mathcal N'/\mathcal  N \cong  \mathfrak a'/\mathfrak a$ and $U=\mathcal N \otimes_\Z \Q$. Associated to the lattice $\mathcal N$ is a vector-valued incoherent Eisenstein series of weight $1$ with Weil representation $\omega_{\mathcal N}$: 
\begin{equation}
E(\tau, s)=E_{\mathcal N}(\tau, s) =\sum_{\mu \in  \mathcal N'/\mathcal N} E(\tau, s, \mu)\phi_\mu.
\end{equation}
Due to the negative sign of its functional equation, $E(\tau, s, \mu)$ vanishes identically at $s = 0$.
Its derivative $E'(\tau, s, \mu)$ at $s = 0$ is a real-analytic modular form of weight 1, and has the following holomorphic part
 (\cite[Proposition 4.6]{BKY}, \cite[Remark 4.3]{YYY})
\begin{equation}
  \label{eq:Ec-hol}
\mathcal E_\Nc(\tau, \mu) = \sum_{m \ge 0} \kappa_\Nc(m, \mu) q^m  ,
\end{equation}
where
for $(m, \mu) \neq (0, 0)$
\begin{equation}
  \label{eq:kappaN2}
  \begin{split}
    \kappa_\Nc(m, \mu)
    &:=
  -2\pi (t\dd)^{-1} \frac{d}{ds}
  \lp \prod_{p <\infty} W_{m, p}(s, \mu; \Nc) \rp\mid_{s = 0},\\
    W_{m, p}(s, \mu; \Nc)
    &:=\int_{\Q_p} \int_{\mu+ \OO_{4D_0, p}} \psi(b\kappa x \bar x)  \psi(-mb)|a(wn(b))|_p^s  dxdb.
  \end{split}
\end{equation}
%
The constant $t\dd$ comes from the normalization 
$$
\hbox{Vol}(\mathcal N_p,  d_{\mathcal N}x) = |\kappa|_p \hbox{Vol}(\mathcal N_p, dx).
$$
and $\prod_{p < \infty} |\kappa|_p = \prod_{p < \infty} |t\dd|_p = (t\dd)^{-1}$. 
  On the other hand, $\kappa_\Nc(0, 0)$ is a suitable constant such that $E'(\tau, 0, \mu) - \mathcal E_\Nc(\tau, \mu) - \frac{\phi_\mu(0)}2 \log v$ decays exponentially as $v \to \infty$.

Assume that $ m >0$, let $\hbox{Diff}(m,  \mathcal N)$ be the set of finite primes $p$ such that $ U_p = \mathcal N \otimes_\Z \Q_p$ does not represent $m$, i.e.,  $\epsilon_{p}(\kappa m) =-1$. Then  $|\hbox{Diff}(m, \mathcal N)| \ge 1$ is odd, and for every $p \in  \hbox{Diff}(m, \mathcal N)$, we have $W_{m, p}(0, \mu, \mathcal N) =0$. 
The coefficient $\kappa_\Nc(m, \mu)$ is then non-zero only if $|\Diff(m, \Nc)| = 1$.

For convenience later, we also denote for $M \in \Nb$ and $(m, \mu) \neq (0, 0)$
\begin{equation}
  \label{eq:kappaNp}
  \begin{split}
      E^{(M)}_m(s, \mu; \Nc) &:=  \prod_{p \nmid M \infty} W_{m, p}(s, \mu; \Nc), \\
      \kappa^{(M)}_\Nc(m ,\mu) &:=
      \frac{d}{ds} \lp E^{(M)}_m(s, \mu; \Nc) \rp \mid_{s = 0} =
       \frac{d}{ds}
\lp \prod_{p <\infty,~ p \nmid M} W_{m, p}(s, \mu; \Nc) \rp\mid_{s = 0},
  \end{split}
\end{equation}
which is independent of the local component of $\mu$ at places dividing $M$.

Let $\rho(n)$. $\rho_p(n)$,  and $\rho^{(M)}(n)$ be as in the introduction, and define 
for $\ell \nmid 6D$,
\begin{equation}
  \label{eq:rho'}
    \rho'_p(m)
    :=  2\sum_{j \ge 1} \rho_p(m/p^{2j-1}) 
    =
    \begin{cases}
          \sum_{j \ge 1} \rho_p(m/p^j) =
  o_p(m) + 1
      &\text{if } p \in \Diff(m, \Nc_1),\\
      0 &\text{otherwise.}
    \end{cases}
\end{equation}
 The following results are well known (see for example \cite{KRYtiny}, \cite[Theorem 2.4]{KY10}) and can also be obtaining by  by  specializing the formulas in Section~\ref{subsec:Whitt}  to $\hat{\mathcal N}$, $\Delta =D_0$ and $\kappa =- dt u $, $u \in \hat\Z^\times$ with  $ u a \in \Nm_{\kay/\Q}\A_{\kay, f}^\times$.

 \begin{proposition} \label{prop:unramified} Let the notation be as above and suppose that $D_1$ and $D_2$ satisfy \eqref{eqd}.
   When $6D \mid M$, for $\mu \in \mathcal N'/\mathcal N$
we have
$$
E_m^{(M)}(s, \mu; \Nc ) =
\frac{1}{L^{(M)}(s+1, \epsilon)}
\begin{cases}
  \prod_{p\nmid M} \sum_{0\le n \le o_p(m)} (\epsilon (p) p^{-s})^n & m \neq 0,\\
\frac{  L^{(M)}(s, \epsilon)}2& m = 0.
\end{cases}
$$
In particular for $m > 0$, we have 
$$
E_m^{(M)}(0, \mu; \Nc )
=
\frac{\sqrt{|D_\kay|}}{\pi h_\kay}  \prod_{p \mid M} L_p(1, \epsilon)  \rho^{(M)}(m).
$$
It is zero if and only if there exists a prime $\ell \nmid M$ such that $\rho_\ell(m) =0$, i.e., $\ell \in \Diff(m, \mathcal N )$. In such a case, its derivative is given by
$$
E_m^{(M), \prime}(0, \mu; \Nc ) =
\frac{\log \ell}{L^{(M)}(1, \epsilon)}
\rho^{(M\ell)}(m)   \frac{1+ o_\ell(m)}2
= \frac{\log \ell}{L^{(M)}(1, \epsilon)}
\sum_{j \ge 1} \rho^{(M)}(m/\ell^j).
$$
\end{proposition}
\begin{remark}
  \label{rmk:muindep}
  Since $E^{(M)}_m(s, \mu; \Nc )$ only depends on $\kay$ and $M$ when $6D \mid M$, we will simply write $E^{(M)}_m(s) = E^{(M)}_m(s, \mu; \Nc )$ in that case.
\end{remark}

 \subsection{Small CM value formula} Let $Z=(z_1, z_2) $ be a small CM point on $X_\dd'$ as in section \ref{subsec:smallCM}, and let $U_\R = U \otimes \R$ be the associated negative plane as in \eqref{eq:Uab}. Let $Z(U) = \{ Z^\pm \} \times T(\Q)  \backslash T(\A_f)/K_{\dd, T}$ be the associated CM cycle with $T$ the torus in \eqref{eq:torus2} and $K_{\dd, T} \subset T(\Ab_f)$ the compact subgroup in Lemma \ref{lemma:KdT}.

 Suppose $D_j$ satisfy \eqref{eqd} and $\af_j = a_j\Z[z_j]$ are $\OO_{D_j}$-ideals with $z_1, z_2$ chosen as in Lemma~\ref{lem:smallCM}.
 Then 
 $$
 \lp \ff_2(z_1)^{24/s} -  \ff_2(z_2)^{24/s} \rp^{s}
 =
  \lp f(\af_1)^{24/s} - f(\af_2)^{24/s} \rp^{s}
  $$
  only depends on the classes $\Ac_j$ of $\af_j$ in $\Cl(D_j)$ for every $s \mid 24$.
  In this case, we denote
  \begin{equation}
    \label{eq:Psidaf}
    \Psi_\dd(\Ac_1, \Ac_2) := \Psi_\dd(z_1, z_2) =
    \prod_{s \mid \dd}   \lp \ff_2(z_1)^{24/s} -  \ff_2(z_2)^{24/s} \rp^{s \cdot \mu(\dd/s)}.
  \end{equation}
  for $\dd \mid 24$.
  Using Lemma \ref{lemma:KdT}, we can express the value of $\Psi_\dd$ at the small CM point $Z(U)$ in the following way.
  \begin{lemma}
\label{lemma:classgp}
    Let the notation be as above.
    Suppose $\Ac_1 \neq \Ac_2$, which is automatic if $D_1 \neq D_2$, then we have
$$
 \sum_{\Bc \in \Cl(D) } \log| \Psi_\dd (\Bc \Ac_1, \Bc \Ac_2)
|=\frac{1}{2}\log |\Psi_\dd(Z(U))|,
$$
where $\Bc \in \Cl(D)$ is viewed as a class in $\Cl(D_j)$ via the natural surjection $\Cl(D) \to \Cl(D_j)$.
\end{lemma}
The following theorem belongs to Schofer \cite{Schofer} (see also \cite[Theorem 1.2]{BY09} for generalization)

\begin{theorem}
\label{thm:Schofer}
  Let the notation be as above, and let
$$
\theta_{\mathcal P}(\tau) = \sum_{\substack{ \mu \in \mathcal P'/\mathcal P\\m\geq0}} a_{\mathcal P}(m, \mu) q^m \phi_\mu
$$
be the holomorphic weight $1$ vector-valued modular form of $(\SL_2(\Z),  \omega_{\mathcal P})$ associated to $\mathcal P$, where
$$
a_{\mathcal P} (m, \mu) := |\Pc_{m, \mu}|=|\{ \alpha \in  \mu + \Pc:\,  Q_\Pc(\alpha)=m\}|.
$$
Suppose $\Ac_1 \neq \Ac_2$, then we have
\begin{equation}
  \label{eq:Schofer}
- \sum_{\Ac \in \Cl(\OO_D) } \log| \Psi_\dd ({\Ac \Ac_1}, {\Ac \Ac_2})|
  = \frac{1}{4}h_D \mathrm{CT}\langle\tilde F_\dd, \theta_{\mathcal P} \mathcal E_{\mathcal N}\rangle
  = \frac{1}{4}h_D
  \sum_{\substack{m, \tm \in \Q,\\  m + \tm ={1}\\m,\tm\geq0}}  C_\dd(\tm, m),
    \end{equation}
    where $h_D $ is the ring class number of $\OO_D$ given by \cite[Theorem 7.24]{cox}
    \begin{equation}
      \label{eq:hD}
      \frac{h_D}{\sqrt{|D|}} = \frac{h_\kay}{ \sqrt{|D_\kay|}} \prod_{p \mid D} L_p(1, \epsilon)^{-1}
    \end{equation}
    and
    \begin{equation}
      \label{eq:Cdm}
      \begin{split}
          C_\dd(\tm, m)
          &:=
 \sum_{\mu \in \mathcal N'_1/\dd \mathcal N,~  \tmu \in  \mathcal P'_1/\dd \mathcal P}
\phi_\dd\lp \frac{\tmu}{\dd}, \frac{\mu}{\dd}\rp    a_{\mathcal P_\dd}\lp \frac{\tm}{\dd},  \frac{\tmu}{\dd}\rp \kappa_{\mathcal N_\dd}\lp \frac{m}{\dd}, \frac{\mu}{\dd}\rp.
      \end{split}
    \end{equation}
\end{theorem}

\begin{proof}
  By Theorem \ref{theo:BorcherdsLifting},   $\Psi_{d}(z_{1},z_{2})$ is a Borcherds lift of $\tilde F_d$ on $X_{d}'$. Then applying \cite[Corollary~1.2]{Schofer} or \cite[Theorem~1.3]{BY09} to $\Phi_{d}(z_{1},z_{2})$ over the CM cycle $Z(U)$ and using the isomorphisms $\mathcal{N}_{1}'/d\mathcal{N}\cong \mathcal{N}'/\mathcal{N}$ and $\mathcal{P}_{1}'/d\mathcal{P}\cong\mathcal{P}'/\mathcal{P}$, one obtains the first equality in \eqref{eq:Schofer}.
  The second follows from the description of the principal part of $\tilde F_\dd$ in Theorem \ref{theo:BorcherdsLifting}. Note that the constant term of $\tilde F_\dd$ is trivial unless $\dd = 1$. In that case, one can use Proposition \ref{prop:Lattice} to check that the constant term of $\theta_\Pc \mathcal E_\Nc$ at the coset $L + \smat0{1/2}00$ is trivial. So the only
 the first Fourier coefficient of $\theta_\Pc \mathcal E_\Nc$  contributes to  $ \mathrm{CT}\langle\tilde F_\dd, \theta_{\mathcal P} \mathcal E_{\mathcal N}\rangle$.
\end{proof}


\begin{lemma}
  \label{lemma:Cdmmprod}
  For any $\dd \mid 24$ and $\tm, m \in \Qb_{\ge 0}$ such that $\tm + m = 1$, the quantity $C_\dd(\tm, m)$ defined in \eqref{eq:Cdm} can be expressed as
  \begin{equation}
    \label{eq:keyeq}
    C_\dd(\tm, m) =
-2\pi \frac{d}{ds} \lp     E^{(6D)}_{m}\lp s\rp
\sum_{\substack{\alpha \in \frac{1}{\sqrt{D}}\taf_0\\ \frac{t}{a}\Nm(\alpha) = \tm}}
    \prod_{p \mid 6D}\delta_p\lp s, m, {\alpha}{}; \dd_p \rp \rp \mid_{s = 0}
    ,
  \end{equation}
  where 
  for $p \mid 6D$
  \begin{equation}
    \label{eq:deltas}
    \begin{split}
  \delta_p(s, m, {\alpha}; \dd_p) &:=
p^{-o_p(t\dd)}    \sum_{\mu \in \frac{1}{ \sqrt{D}} \af_0 /\dd_p \af \otimes \Zb_p}    \phi_{\dd_p, p}\lp \frac{{\alpha}}{\dd_p}, \frac{\mu}{\dd_p}\rp
  W_{m/\dd_p, p}\lp s, \frac{\mu}{\dd_p}; \Nc_{\dd_p}\rp.
    \end{split}
  \end{equation}
  In particular for $m > 0$, $C_\dd(\tm, m) = 0$ if $\#\Diff(m, \Nc_1) \neq 1$.
\end{lemma}
\begin{proof}
  First of all, we can apply Remark \ref{rmk:Psid} and \eqref{eq:kappaN2} to write
  \begin{align*}
    C_\dd(\tm, m)
    &=   \sum_{\mu^{} \in \frac{1}{2\sqrt{D}}\af/\dd \af}
      \sum_{\substack{\alpha \in \frac{1}{2\sqrt{D}} \taf\\ \frac{t}{a}\Nm(\alpha) = \tm}}
        \phi_{\dd}\lp \frac{\alpha}{\dd}, \frac{\mu}{\dd}\rp
    \kappa^{}_{\mathcal N_\dd}\lp \frac{m}{\dd}, \frac{\mu}{\dd}\rp\\
    &=   -2\pi (t\dd)^{-1} \frac{d}{ds} \lp
      \sum_{\substack{\alpha \in \frac{1}{2\sqrt{D}} \taf\\ \frac{t}{a}\Nm(\alpha) = \tm}}
    \prod_{p < \infty}
     \sum_{\mu^{} \in \frac{1}{2\sqrt{D}}\af/\dd \af \otimes \Zb_p}
    \phi_{\dd, p}\lp \frac{\alpha}{\dd}, \frac{\mu}{\dd}\rp
    W_{m, p}\lp s, \frac{\mu}{\dd}; \Nc_\dd\rp
    \rp \mid_{s = 0}
  \end{align*}
  If $p \nmid 6D$, then $\frac{1}{2\sqrt{D}}\af/\dd \af \otimes \Zb_p$ is trivial, $W_{m, p}(s, \mu; \Nc_\dd)$ is independent of $\mu$ and we obtain
  $$
    \prod_{p \nmid 6D}
     \sum_{\mu^{} \in \frac{1}{2\sqrt{D}}\af/\dd \af \otimes \Zb_p}
    \phi_{\dd, p}\lp \frac{\alpha}{\dd}, \frac{\mu}{\dd}\rp
    W_{m/\dd, p}(s, \mu; \Nc_\dd)
    = E^{(6D)}_m(s).
    $$
    If $p \mid 6D$,
    then $\frac{1}{2\sqrt{D}}\af/\dd \af \otimes \Zb_p =\frac{1}{2\sqrt{D}}\af/ \dd_p \af \otimes \Zb_p    $ and
    $\phi_{\dd, p}(\tfrac{\alpha}{\dd}, \frac{\mu}{\dd}) = \phi_{\dd_p, p}(\frac{\alpha}{\dd_p}, \frac{\mu}{\dd_p})$ by Remark \ref{rmk:Psid} as $\dd \mid 24$.
    Furthermore, the definition of $W_{m, p}$ in \eqref{eq:kappaN2} directly implies that
    $W_{m/\dd, p}(s, \mu/\dd; \Nc_\dd) = W_{m/\dd_p, p}(s, \mu/\dd_p; \Nc_{\dd_p})$.
Finally from \eqref{eq:supp}, we can replace $\frac{\af}{2}$ and $\frac{\taf}2$ by ${\af_0}$ and ${\taf_0}$ respectively. 
Putting these together then finishes the proof.
 \end{proof}
    For convenience, we define
    \begin{equation}
  \label{eq:delta6}
  \begin{split}
      L_p(s+1, \epsilon)
    \frac{\delta_p(s, m, \alpha; \dd_p)}{L_p(s, \epsilon)^{\delta_{m = 0}}}
    &= 
    \delta_p(m, \alpha; \dd_p) + 
    \delta_p'(m, \alpha; \dd_p) (\log p \cdot s) + O(s^2)
  \end{split}
\end{equation}
when $p \mid 6$, and
\begin{equation}
  \label{eq:deltaD}
  \begin{split}
          p^{o_p(D)/2}
    \frac{\delta_p(s, m, {\alpha}; 1)}{L_p(s, \epsilon)^{\delta_{m = 0}}}
&=
  \delta_p(-D_0 t m , \sqrt{D}{\alpha})
  +
    \delta_p'(-D_0 t m,\sqrt{D}{\alpha})
  (\log p \cdot s) + O(s^2)
  \end{split}
\end{equation}
when $p\mid D$.
%
We can now expand the expression for $C_\dd(\tm, m)$ in \eqref{eq:keyeq}
in terms of the quantities in \eqref{eq:delta6} and \eqref{eq:deltaD} using product rule as follows.
  \begin{proposition}
    \label{prop:Diff1}
For $m \ge 0$, define
    \begin{equation}
      \label{eq:Cdmm}
      \begin{split}
C_{\dd, \ell}(\tm, m)
:=        -\frac{2}{{h_D} }
&\rho^{(6D\ell)}(m)
\sum_{\substack{\tilde{\alpha} \in \taf_0\\ \frac{1}{a}\Nm(\tilde{\alpha}) = -D_0t\tm}}
\prod_{p \mid 6,~ p \neq \ell} \delta_p\lp m, \frac{\tilde{\alpha}}{\sqrt{D}}; \dd_p\rp\\
&\cdot \prod_{p \mid D,~ p \neq \ell} \delta_p(-D_0 t m, \tilde{\alpha})
      \begin{cases}
\rho'_\ell(m),        & \text{ if } \ell \nmid 6D,\\
\delta'_{\ell}(-D_0 tm , \tilde{\alpha}), & \text{ if } \ell \mid D, \\
\delta'_{3}(m, \tilde{\alpha}/\sqrt{D}; \dd_3), & \text{ if } \ell = 3.
      \end{cases}
      \end{split}
    \end{equation}
    where 
    $\delta_p$ and $\delta_\ell'$ are defined in \eqref{eq:delta6} and \eqref{eq:deltaD}.
    Then
    \begin{equation}
      \label{eq:Cdtmm}
      C_{\dd}(\tm, m) = \sum_{\ell < \infty} C_{\dd, \ell}(\tm, m) \log \ell.
    \end{equation}
  \end{proposition}
  \begin{proof}
    This follows directly from Propositon\ \ref{prop:unramified} and Lemma \ref{lemma:Cdmmprod}.
  \end{proof}

  The values of $\delta_p$ and $\delta'_p$ will be computed in Section \ref{sect:local}, from which it is clear that for fixed $\dd, \tm, m$, there is at most one $\ell$ such that $C_{\dd, \ell}(\tm, m) \neq 0$.

\section{Statement and Proof of Main Result}  \label{sect:main}

\subsection{The main formula} Now we are ready to state and  prove the main general formula of this paper. 
  
  \begin{theorem}
  \label{thm:main}
  Let $s \mid 24$ and $\Ac_i \in \Cl(D_i)$ with $D_i$ any discriminant satisfying \eqref{eqd} such that $\sqrt{D_1D_2} \in \Zb$.
  Denote $s' := \gcd(s, 3^{1 - (\tfrac{D}{3})})$, 
  $D_0 := \mathrm{gcd}(D_1, D_2) < 0$, $D := \mathrm{lcm}(D_1, D_2) = D_0 t^2$ for some $t \in \Nb$. 
  Recall that $\taf_0$ is the $\mathcal{O}_{D_0}$ integral ideal defined in \eqref{eq:af}, and represents the class of $\mathfrak{A}_1^{-1}\mathfrak{A}_2 \in \Cl(D_0)$. 
  Then for a prime number $ \ell$,

\begin{equation}
    \label{eq:main}
    \begin{split}
          &2\ord_{\ell} \lp \prod_{\Bc \in \Cl(D)} (f(\Ac_1 \Bc)^{24/s} - f(\Ac_2 \Bc)^{24/s}) \rp
          \\
          &=   \sum_{\substack{n, \tn\in\Z_{\ge 0} \\n + \tn = -D_0 t}}
     \sum_{\substack{r \mid (s/s_\ell),~ r>0\\ (s'/s'_\ell) \mid r\\ \frac{n\tn}{4r^2} \equiv 19 \bmod{\frac{s/s_\ell}{r}}}}
  \sum_{AB = 2r}
  \rho^{(D\ell)}\lp \frac{n}{A^2}\rp
 \sum_{\substack{\tilde{\alpha} \in \taf_0\\ \frac{\Nm(\tilde{\alpha})}a = \frac{\tn}{B^2}}}
   \prod_{p \mid D \text{ prime,} p \neq \ell}   \delta_{p}(n, \tilde{\alpha})
    \cdot
  \begin{cases}
    \rho'_{\ell}(n),& \ell \nmid 3D,\\
    \delta'_{\ell}(n, \tilde{\alpha}),& \ell \mid D,\\
    \rho'_{3s_3}(n),& \ell = 3,
  \end{cases}
    \end{split}
\end{equation}
 where $\delta_p, \delta_p'$ are defined in \eqref{eq:delta6} and \eqref{eq:deltaD}, and $\rho^{(M)}(x)$ is given by (\ref{eq:rhom}).  In particular, this formula depends only on the class of $\taf_0$ in $\Cl(D)$.
  \end{theorem}
  \begin{proof}
Let $z_i \in \H$ be CM points associated to $\Ac_i$,  satisfying Lemma \ref{lem:smallCM}.
    By \eqref{eq:Blift}, Lemma \ref{lemma:classgp} and Theorem \ref{thm:Schofer}, we can write
$$
s \log \left|
 \prod_{\Bc \in \Cl(D)} (f(\Ac_1 \Bc)^{24/s} - f(\Ac_2 \Bc)^{24/s})
\right| =
-\frac{1}{4}h_D
\sum_{\substack{m, \tm \in \Q_{\ge 0},\\  m + \tm ={1}}}
\sum_{\dd \mid s}
C_\dd(\tm, m),
$$
where $C_\dd(\tm , m)$ is defined in \eqref{eq:Cdm} for $\af_i = a_i \Z[z_i]$.
By Prop.\ \ref{prop:Diff1}, we can then write
\begin{align*}
  2\ord_{\ell}& \lp \prod_{\Bc \in \Cl(D)} (f(\Ac_1 \Bc)^{24/s} - f(\Ac_2 \Bc)^{24/s}) \rp
   =
- \frac{h_D}2    \sum_{\substack{m, \tm \in \Q_{\ge 0},\\  m + \tm ={1}}}
\sum_{\dd \mid s}
C_{\dd, \ell}(\tm, m)
\end{align*}

By the calculations in Section \ref{sect:local} (in particular Prop.\ \ref{prop:s2contr}, \ref{prop:s3contr} and \ref{prop:dpcontr}), we see that $C_{\dd, \ell}(\tm, m)$ vanishes when $n, \tn \not\in \Z$ or $\tilde{\alpha} \not\in \taf_0$, where we denote
$$
n:= - D_0 t m,~ \tn := -D_0 t \tm.
$$
In particular, Propositions \ \ref{prop:s2contr} and \ref{prop:s3contr} give us
\begin{align*}
 & \frac{-s^{-1} h_D}{2} \sum_{\dd \mid s} C_{\dd, \ell}\lp -\frac{\tn}{D_0t}, -\frac{n}{D_0t} \rp
\\
&=
\rho^{(6D\ell)}(n)
\sum_{\substack{\tilde{\alpha} \in \taf_0\\ \frac{1}{a}\Nm(\tilde{\alpha}) = \tn}}
  \prod_{p \mid 6,~ p \neq \ell}
  \sum_{\substack{s'_p \mid r_p \mid s_p\\ \frac{n\tn}{r_p^2} \equiv 19 \bmod{\frac{s_p}{r_p}}}}
  \sum_{AB = 2r_p} \rho_p \lp \frac{n}{A^2} \rp  \mathds{1}_{\taf_0} \lp \frac{\tilde{\alpha}}{B} \rp\\
 &\quad\times
  \prod_{p \mid D,~ p \neq \ell} \delta_p(n, \tilde{\alpha})
      \begin{cases}
        \rho'_{\ell} (n),& \text{ if } \ell \nmid 6D,\\
\rho'_{3s_\ell}(n),& \text{ if } \ell = 3,\\              \delta'_{\ell}(n, \tilde{\alpha}), & \text{ if } \ell \mid D.
      \end{cases}\\
&=  
  \sum_{\substack{r \mid (s/s_\ell)\\ (s'/s'_\ell) \mid r\\ \frac{n\tn}{r^2} \equiv 19 \bmod{\frac{s}{r}}}}
  \sum_{AB = 2r}
\rho^{(D\ell)}\lp \frac{n}{A^2} \rp
 \sum_{\substack{\tilde{\alpha} \in B\taf_0\\ \frac{1}{a}\Nm(\tilde{\alpha}) = -D_0 t\tm}}
  \prod_{p \mid D,~ p \neq \ell} \delta_p(n, \tilde{\alpha})
      \begin{cases}
        \rho'_{\ell} (n),& \text{ if } \ell \nmid 6D,\\
\rho'_{3s_\ell}(n),& \text{ if } \ell = 3,\\
        \delta'_{\ell}(n, \tilde{\alpha}), & \text{ if } \ell \mid D.
      \end{cases}
\end{align*}
By Proposition \ \ref{prop:dpcontr}, we have $\delta_p^{(')}(n, {\tilde{\alpha}}) = \delta_p^{(')}(n, {\tilde{\alpha}}/B)$ for all $\tilde{\alpha} \in B \taf_0$ and $p \mid D$. Summing over all $n \in \Z_{\ge 0}$ then gives \eqref{eq:main}.
\end{proof}

\begin{corollary}\label{cor:main}
    Any prime factor of ${\rm disc}(D,s,[\mathfrak{a}])$ is no larger than~$|D|$.
\end{corollary}

\begin{proof}
    Notice by~\eqref{eq:main} that the argument~$n$ is no larger than~$|D|$. Also by the definition of $\rho_{\ell}'(n)$, one can tell that $\rho_{\ell}'(n)=0$ for $\ell>n$. Therefore, one must have $\ell\leq |D|$.
\end{proof}

\subsection{The Yui-Zagier conjecture on discriminant}
Throughout this section, we let $D< 0$ be a fundamental discriminant satisfying \eqref{eqd}. 
For any $\OO_D$-integral ideal $\af$, let $c(\af)$ denote the content of $\af$, i.e.\ largest positive integer such that $\af/c(\af)$ is still an $\OO_D$-integral ideal. For an integer $n$, we define 
\begin{equation}
    S(D, n) =\{ p < \infty \text{ prime}:\,  \epsilon_p(n) =-1 \}.
\end{equation}
Here $\epsilon_p$ is the quadratic character of $\Q_p^\times$ associated to $\kay_{D, p}$.  Notice that when $n >0$,  $S(D, -n)$ has odd cardinality and equals to $\Diff(-n/(D_0t), \Nc_1)$ from Section \ref{subsec:incoh} with $D = D_0$ and $t = 1$. 
We now define
\begin{equation}
  \label{eq:wln}
  w(\ell, n) =
    w_{D}(\ell, n) :=
  \begin{cases}
  \sigma(  \hbox{gcd}(n, |D|/\hbox{gcd}(\ell, |D|))), &\text{if }S(D, -n) =\{\ell\},\\
    0 &\text{otherwise,}
  \end{cases}
\end{equation}
where $\sigma$ is the divisor function $\sigma(n)=\sum_{d|n}1$. 
As a function in $n$, $w(\ell, n)$ is in fact defined on $(\Qb^\times \cap \prod_{p \mid D} \Zb_p^\times)/(\Nm(\kay^\times) \cap \prod_{p \mid D} \Zb_p^\times) $. 
\begin{theorem}
  \label{thm:YZconj}
Let $D< 0$ be a fundamental discriminant satisfying~\eqref{eqd}, and for $s \mid 24$ 
  and each non-trivial class $\tAc=[\tilde{\mathfrak{a}}] \in \Cl(D)$ with ${\rm Nm}(\tilde{\mathfrak{a}})=a$ coprime to~$D$,  denote
  \begin{equation}
    \label{eq:discA}
\disc(D; s, \tAc) := \prod_{\Ac \in \Cl(D)} (f(\Ac)^{24/s} -  f(\Ac\tAc)^{24/s}) \in \kay_{D},    
  \end{equation}
so that the discriminant ${\rm disc}(D;s)$ of $f(\mathcal{O}_{D})^{24/s}$ is given by
$$
{\rm disc}(D;s)=\prod_{\substack{\tilde{\mathfrak{A}}\in{\rm Cl}(D)\\\tilde{\mathfrak{A}}\ne [\mathcal{O}_{D}]}}{\rm disc}(D;s,\tilde{\mathfrak{A}}).
$$
Then the following hold.
\begin{enumerate}
    \item For each prime $\ell$ split in $k_{D}$, we have ${\rm ord}_{\ell}({\rm disc}(D;s,\tilde{\mathfrak{A}}))=0$.
 \item  For each prime $\ell\ne3$ inert in $\kay_{D}$, 
$\ord_\ell(\disc(D; s, \tAc))=a_\ell $ 
is equal  to the number of pairs of integral ideals $(\mathfrak{b}_{1}, \bfrak_{2})$ weighted by $w(\ell,(|D|-{\rm Nm}(\mathfrak{b}_{2}))a)$ such that
 $[\mathfrak{b}_{2}]=\tilde{\mathfrak{A}}$, $\ell^{j}{\rm Nm}(\mathfrak{b}_{1})+{\rm Nm}(\mathfrak{b}_{2})=|D|$ for some $j>0$, and that $\mathfrak{a} =(2s')^{-1}\mathfrak{b}_{1}\mathfrak{b}_{2}$ is a  nonzero integral ideal $\mathfrak{a}$ satisfying that
    $$
    c(\mathfrak{a})\left(\ell\frac{{\rm Nm}(\mathfrak{a})}{c(\mathfrak{a})^{2}}+5\right)\equiv0\pmod{\frac{s}{s'}},
    $$
    where $s'=\gcd\left(s,3^{1-\left(\frac{D}{3}\right)}\right)$.

\item For $\ell=3$ inert in $k_{D}$, we have that
\begin{align*}
&{\rm ord}_{3}({\rm disc}(D;s,\tilde{\mathfrak{A}}))\\
   & =\frac{1}{2}   \sum_{\substack{n\in\mathbb{N}\\n + \tn =|D|\\\tilde{n}\geq0}}w(3,na)
     \sum_{\substack{r \mid (s/s_3),~ r>0\\ \frac{n\tn}{4r^2} \equiv 19 \bmod{\frac{s/s_3}{r}}}}
  \sum_{\substack{AB = 2r\\j \ge 1}}
\lp
  \rho^{}\lp \frac{n}{3^j A^2}\rp
  +
  \rho^{}\lp \frac{n}{3^j s_3 A^2}\rp
  \rp
 r_{\tAc}\lp \frac{\tn}{B^2} \rp,
\end{align*}
where $\rho^{(M)}(n)$ is defined as in~\eqref{eq:rhom},  $s_{3}$ is the exact power of~$3$ dividing~$s$.

        \item For $\ell|D$, we have
          $$
          \ord_\ell(\disc(D; s, \tAc))=
           \frac{1}{2}w(\ell,|D|a)
          + a_\ell. $$ 
Here  $a_\ell$ is  defined as in part~(2).
   

\end{enumerate}

\end{theorem}
\begin{proof}
Take fundamental discriminants 
 $D=D_{1}=D_{2}$ in Theorem~\ref{thm:main}. 
 First, notice that the term $n = 0, \tn = -D$ is empty, since $\tilde\af_0$ is not principal and there is no $\tilde\alpha \in \tilde\af_0$ with norm $a\tn$. 
Now,  using Prop.\ \ref{prop:dpcontr} and Remark \ref{rmk:chip} with $t = 1, r = r_0 = 1$,  one can find that for $p|D$ and $n$ a positive integer,
\begin{align}\label{ddn}
  \delta_p(n, \tilde{\alpha})
  &=
        \begin{cases}
  2& \mbox{if  $\epsilon_{p}(-na)=1$ and $o_p(n) \ge 1$},\\
  1& \mbox{if  $\epsilon_{p}(-na)=1$ and $o_p(n) = 0$},\\
  0& \mbox{if  $\epsilon_{p}(-na)=-1$},
\end{cases}\\
\label{ddn2}
\delta_p'(n, \tilde{\alpha})&=\begin{cases}
    o_p(n)
    &\mbox{if $\epsilon_{p}(-na)=-1$,}\\
    0&\mbox{otherwise.}
\end{cases}
\end{align}
In particular, $\delta'_p(n, \tilde\alpha) = \rho'_\ell(n/D_0)$.
Then let $\ell$ be a prime. If $\ell$ is split, it is clear that by~\eqref{eq:rho'}, the formula~\eqref{eq:main} implies that ${\rm ord}_{\ell}({\rm disc}(D;s,\tilde{\mathfrak{A}}))=0$.  This takes care Item (1).

If $\ell$ is either inert or ramified, it is not hard to see that
$$
\prod_{p|D\,\,prime,\,p\ne\ell}\delta_{p}(n,\tilde{\alpha})=w(\ell,(|D|-\tilde{n})a),
$$
where $w(\ell,na)$ is defined as before Theorem~\ref{thm:YZconj}.

Item (2) and (4):
We  apply~\eqref{eq:main} together with~\eqref{ddn} and~\eqref{eq:rho'} to deduce that
     \begin{align}\label{ordell}
             & \ord_{\ell} \disc(D; s, \tAc)\\
              &=\frac12\sum_{\substack{n > 0,\tilde{n}\ge 0\\n + \tn = -D}}
w(\ell,(|D|-\tilde{n})a)  
\sum_{j \ge 1}
     \sum_{\substack{r \mid s,~ s' \mid r\\ \frac{n\tn}{4r^2}  \equiv 19 \bmod{\frac{s}{r}}}}
\sum_{AB = 2r}
      \rho_{}\lp \frac{n}{A^2 \ell^j}\rp  \sum_{\substack{\tilde{\alpha}\in\tilde{\mathfrak{a}}_0\\{\rm Nm}(\tilde{\alpha})/a=\tilde{n}/B^{2}}}1\nonumber\\
          &=   
\sum_{\substack{n>0,\tilde{n}\ge0\\n + \tn = -D}}
w(\ell,(|D|-\tilde{n})a)   
\sum_{j \ge 1}
     \sum_{\substack{r \mid s,~ s' \mid r\\ \frac{n\tn}{4r^2}  \equiv 19 \bmod{\frac{s}{r}}}}
\sum_{AB = 2r}
       \rho_{}\lp \frac{n}{A^2 \ell^j}\rp
       \sum_{\substack{\bfrak \subset \Oc_{D_0}\\{\rm Nm}(\mathfrak{b})=\tilde{n}/B^{2}\\ 
 [\mathfrak{b}]=\tilde{\mathfrak{A}}}}1\nonumber\\
 &=\sum_{\substack{n>0,\tilde{n}\ge0\\n + \tn = -D}}
w(\ell,(|D|-\tilde{n})a)     
\sum_{j \ge 1}
      \sum_{\substack{r \mid s,~ s' \mid r\\ \frac{n\tn}{4r^2}+5  \equiv 0 \bmod{\frac{s}{r}}}}
\sum_{\substack{{\rm Nm}(\mathfrak{b}_{1})=n/\ell^{j}\\{\rm Nm}(\mathfrak{b}_{2})=\tilde{n}\\  [{\mathfrak{b}}_{2}]=\tilde{\mathfrak{A}}\\\mathfrak{b}_{1}\mathfrak{b}_{2}=2r\mathfrak{a}}}1.\nonumber
     \end{align}
     Note for $\tn = 0, n = -D$, we have $\rho(n/A^{2}\ell^j)=0$ for  $\ell \nmid D$ and all $j \ge 1$.
     Otherwise for $ \ell \mid D$, this term gives $\frac{1}{2}w(\ell,|D|a)$. 
     Also, we can delete $a$ from $w(\ell, (|D|-\tilde n)a)$ by choosing representative $\tilde\af$ of the class $\tilde\Ac$ having norm co-prime to $\ell$. 

     Following the reasoning given in the remarks above \cite[Eq. ($14_{?}$)]{YZ}, the rest of the summation amounts to counting the pairs of integral ideals $(\mathfrak{b}_{1},\mathfrak{b}_{2})$ weighted by $w(\ell,(|D|-{\rm Nm}(\mathfrak{b}_{2}))a)$ such that $[\mathfrak{b}_{2}]=\tilde{\mathfrak{A}}$, $\ell^{j}{\rm Nm}(\mathfrak{b}_{1})+{\rm Nm}(\mathfrak{b}_{2})=|D|$ for some $j>0$ and $\mathfrak{b}_{1}\mathfrak{b}_{2}=2s'\mathfrak{a}$ for some nonzero integral ideal $\mathfrak{a}$ satisfying
    $$
    c(\mathfrak{a})\left(\ell\frac{{\rm Nm}(\mathfrak{a})}{c(\mathfrak{a})^{2}}+5\right)\equiv0\pmod{\frac{s}{s'}}.
    $$
    Item (3) follows by the same argument using~\eqref{eq:main} and~\eqref{eq:s3contr}.   
 \end{proof}

\begin{proof}[Proof of Theorem~\ref{corYZ}]
  Theorem~\ref{corYZ}  follows from Theorem~\ref{thm:YZconj} via noticing that when $|D|$ is a prime, for $\ell$ inert in $k_{D}$, it is clear that ${\rm Nm}(\mathfrak{b}_{2})$ in item (2) of Theorem~\ref{thm:YZconj} must be strictly smaller than~$p$ for the validity of $\ell^{j}{\rm Nm}(\mathfrak{b}_{1})+{\rm Nm}(\mathfrak{b}_{2})=p$, and $w(\ell, (p-{\rm Nm}(\mathfrak{b}_{2}))a)=2^{\omega(\gcd(p-{\rm Nm}(\mathfrak{b}_{2}),p))}=1$. And for $\ell=p$, it is clear that there are no nonzero integral ideals $\mathfrak{b}_{1}$, $\mathfrak{b}_{2}$ such that $p^{j}{\rm Nm}(\mathfrak{b}_{1})+{\rm Nm}(\mathfrak{b}_{2})=p$ for $j>0$, so the quantity $a_{p}$ in item (4) of Theorem~\ref{thm:YZconj} must vanish, and $w(p,pa)=2^{\omega(\gcd(p,1))}=1$, since $\epsilon_{p}(-pa)=-1$ as the class number of $
  \Qb(\sqrt{-p})$ is odd.
\end{proof}



Our formulas in  Theorem~\ref{thm:YZconj} and Theorem ~\ref{corYZ} match with examples in \cite{YZ}, even when $D$ is not a prime. The following exam verify our formula for $f(\mathfrak a)^{24}$ for $=-31$.

\begin{example}
    For $D=-31$, one can numerically find that the defining polynomial of $f\left(\frac{D+\sqrt{D}}{2}\right)^{24}$ over $\mathbb{Q}$ is $X^3-165X^2+9642X-1$, whose discriminant is $-3^{12}11^{2}23^{2}31$. Since ${\rm Cl}(-31)=\{[\mathcal{O}_{31}],[\mathfrak{a}],[\overline{\mathfrak{a}}]\}$, where $\mathfrak{a}=[2,\frac{-1+\sqrt{-31}}{2}]$, and ${\rm ord}_{\ell}({\rm disc}(-31,1,[\mathfrak{a}]))={\rm ord}_{\ell}({\rm disc}(-31,1,[\overline{\mathfrak{a}}]))$, then it suffices to compute ${\rm ord}_{\ell}({\rm disc}(-31,1,[\mathfrak{a}]))$. 
   In addition, by Corollary~\ref{cor:main}, one only needs to compute for $\ell=11,13,17,23,29$. 
   It is routine to check that the only triples $(\ell,\mathfrak{b}_{1},\mathfrak{b}_{2})$ satisfying the assumptions imposed by item (2) of Theorem~\ref{corYZ} are $(11,\mathcal{O}_{k},2\mathfrak{p}_{5})$, where $\mathfrak{p}_{5}$ is a prime ideal over~$5$, and $(23,\mathcal{O}_{k},2\mathfrak{a})$, and these yield that
   $$
   {\rm ord}_{11}({\rm disc}(-31,1,[\mathfrak{a}]))={\rm ord}_{23}({\rm disc}(-31,1,[\mathfrak{a}]))=1,
   $$
   and for $\ell=13,17,29$, ${\rm ord}_{\ell}({\rm disc}(-31,1,[\mathfrak{a}]))=0$.
    For $\ell=31$, by item (4) of Theorem~\ref{corYZ} one has that
\begin{align*}
        {\rm ord}_{\ell}({\rm disc}(-31,1,[\mathfrak{a}]))&=\frac{1}{2}.
    \end{align*}
Also, when $\ell=3$, by item (3) of Theorem~\ref{corYZ} one has that
\begin{align*}
    &{\rm ord}_{3}({\rm disc}(-31,1,[\mathfrak{a}]))\\
        &=\frac{1}{2}\rho^{(31\cdot3)}(3)\sum_{\substack{{\alpha}\in\mathfrak{a}\\{\rm Nm}({\alpha})=19}}\delta_{31}(12)\rho_{3}'(12) + \frac{1}{2}\rho^{(31\cdot3)}(6)\sum_{\substack{{\alpha}\in\mathfrak{a}\\{\rm Nm}({\alpha})=7}}\delta_{31}(24)\rho_{3}'(24)\\
        &\quad+ \frac{1}{2}\rho^{(31\cdot3)}(27)\sum_{\substack{{\alpha}\in\mathfrak{a}\\{\rm Nm}({\alpha})=1}}\delta_{31}(27)\rho_{3}'(27)+ \frac{1}{2}\rho^{(31\cdot3)}(15)\sum_{\substack{{\alpha}\in\mathfrak{a}\\{\rm Nm}({\alpha})=4}}\delta_{31}(15)\rho_{3}'(15)\\
        &\quad + \frac{1}{2}\rho^{(31\cdot3)}(3)\sum_{\substack{{\alpha}\in\mathfrak{a}\\{\rm Nm}({\alpha})=7}}\delta_{31}(3)\rho_{3}'(3)\\
        &=6.
\end{align*}
Putting all these together one finds that
$$
|{\rm disc}(-31,1,[\mathfrak{a}]){\rm disc}(-31,1,[\overline{\mathfrak{a}}])|=3^{12}11^{2}23^{2}31,
$$
which is exactly the absolute value of the discriminant of the aforementioned polynomial.
\end{example}

\subsection{Factorization of Resultants}

The factorization of the resultant is given in the following result, which specializes to Theorem \ref{thm:main2}. 
\begin{theorem}
  \label{thm:main2a}
  Let $D_i = D_0 t_i^2, t, \mathfrak{A}_i, P_i(x)$ and $R(P_1, P_2)$ be the same as in Theorem \ref{thm:main2}. 
  Suppose  all primes dividing $t := t_1t_2 > 1$ are non-split in $\kay$, and denote
  $$
D_{0, t} := \gcd(D_0, t),~ D_0' := D_0/D_{0, t}
  $$
  For a non-split prime   $\ell$,
let $\kappa_\ell$ be any negative integer number co-prime  to $D_0$ such that for all primes $p \mid D_0$, 
$(\tfrac{\kappa_\ell}{p}) = 1$ if and only if $p \neq \ell$. 
If $\ell \nmid 3t$, then $\hbox{ord}_{\ell} R(P_1, P_2)$ is given by
\begin{equation}
    \label{eq:res1}
          \sum_{\substack{n,\tn\in\mathbb{N}\\n + \tn = |D_0 t|\\ \gcd(n, t^2) \mid D_{0, t}}}
{\sigma(\gcd(n, \tfrac{D_0'}{\gcd(\ell, D_0')}))}
     \sum_{\substack{r \mid s,~ r>0\\s' \mid r\\ \frac{n\tn}{4r^2} \equiv 19 \bmod{\frac{s}{r}}}}
  \sum_{\substack{AB = 2r\\j \ge o_\ell(D_0)}}
   \rho^{}\lp \frac{n}{A^2\ell^j}\rp
\rg(\tn/B^2; -\kappa_\ell \cdot n).
\end{equation}
where $s' = \gcd(s, 3^{1-(\tfrac{D}3)})$. 
If $\ell \mid t$,  then $\hbox{ord}_{\ell} R(P_1, P_2)$ is given by
\begin{equation}
    \label{eq:res2}
          \sum_{\substack{n\in\mathbb{N},\tn\geq0\\n + \tn = |D_0 t|\\ \gcd(n^{(\ell)}, t^2) \mid D_{0, t}}}
{\sigma(\gcd(n, {D_0'}))}
     \sum_{\substack{r \mid s,~ r>0\\s' \mid r\\ \frac{n\tn}{4r^2} \equiv 19 \bmod{\frac{s}{r}}}}
  \sum_{AB = 2r}
   \rho^{}\lp \frac{n}{A^2\ell \gcd(\ell, D_{0,t})}\rp
\rg(\tn/B^2; -\kappa_\ell \cdot n)
\end{equation}
plus an extra term $\rho(0)(1-\epsilon(\ell))\rho(-D_0t)$ if $\ell$ is the only prime dividing $t$. 
Here $n^{(\ell)} := n/\ell^{o_\ell(n)}$.
If $\ell = 3$, then $\hbox{ord}_{\ell} R(P_1, P_2)$ is given by
\begin{equation}
  \label{eq:res3}
          \sum_{\substack{n, \tn\in\mathbb{N}\\n + \tn = |D_0 t|\\ \gcd(n, t^2) \mid D_{0, t}}}
\frac{\sigma(\gcd(n, D_0'))}2         
     \sum_{\substack{r \mid s_2,~ r>0\\ \frac{n\tn}{4r^2} \equiv 19 \bmod{\frac{s_2}{r}}}}
  \sum_{\substack{AB = 2r\\j \ge 1}}
\lp  \rho^{}\lp \frac{n}{A^23^{j}}\rp
+   \rho^{}\lp \frac{n/s_3}{A^23^j}\rp\rp
\rg(\tn/B^2; - n).
\end{equation}
\end{theorem}

\begin{proof}
 First we have  the short exact sequence:
  $$
1  \longrightarrow \Cl(D) \stackrel{(\phi_1, \phi_2)}{\longrightarrow} \Cl(D_1) \times \Cl(D_2) \stackrel{}{\longrightarrow} \Cl(D_0)  \longrightarrow 1.~
  $$
Here $\phi_i([\mathfrak a]) =[\mathfrak a \mathcal O_{D_i}]$, and the last map is given by 
$([\mathfrak a_1],  [\mathfrak a_2]) \mapsto [\mathfrak a_1^{-1} \mathfrak a_2 \mathcal O_{D_0}].
$
So  the resultant is given by
    $$
\prod_{\mathfrak{A_1} \in \Cl(D_1)}
\prod_{\Ac_2 \in \Cl(D_2)} (f(\Ac_1)^{24/s} - f(\Ac_2)^{24/s})
=
    \prod_{\Ac \in S_0}
    \prod_{\Bc \in \Cl(D)} (f(\phi_1(\Bc))^{24/s} - f(\Ac \phi_2(\Bc))^{24/s}),
    $$
    where 
    $S_0 \subset \Cl(D_2)$ is a subset such that
    $S_0 \cong \Cl(D_0)$ as sets via the restriction map $\Cl(D_2) \to \Cl(D_0)$. 
    As in \eqref{eq:discA}, we denote
    $$
    \disc(D, s, \Ac) :=
    \prod_{\Bc \in \Cl(D)} (f(\phi_1(\Bc))^{24/s} - f(\Ac \phi_2(\Bc))^{24/s}). 
    $$
    To find $\ord_\ell\disc(D, s, \Ac)$, we apply Theorem \ref{thm:main} and need to evaluate $\delta_p(n, \tilde\alpha)$ and $\delta'_\ell(n, \tilde\alpha)$. 
    For this, we use Proposition \ref{prop:dpcontr}, in particular equations \eqref{eq:dp-1}, \eqref{eq:dp-2}, \eqref{eq:dp0a} and \eqref{eq:dp0b}.

    When $n = 0$, the quantity $\delta_p(0, \tilde\alpha)$ vanishes for $p \mid t$. So if $\ell$ is not the only prime factor of $t$, then the contribution of the terms with $n = 0$ to the right hand side of \eqref{eq:main} is 0.
    Otherwise, the contribution is $\rho(0)(1-\epsilon(\ell))r_{[\tilde\af_0]}(-D_0t)$. Summing over $[\tilde\af_0] \in \Cl(D_0)$ gives the extra contribution to \eqref{eq:res2}. 

When $n > 0$, we suppose        $\Diff(-n/(D_0t), \Nc_1) = \{\ell\}$.
    Note that by assumption, $p \mid D = D_0t^2$ is either ramified or inert.
If $p \mid D_0' \mid D_0$, i.e.\ $p \nmid t$, we have $\delta_p(n, \tilde\alpha) = \sigma(\gcd(n, p))$. 
        If $p \mid D_{0, t}$, then $p$ is ramified and 
    $\tilde\alpha\in p\Oc_{D_0}$ if and only if $o_p(n) = o_p(aD_0t + \Nm(\tilde\alpha)) \ge 2$.
    Similarly for $p \mid t$ but $p \nmid D_{0, t}$, we have     $\tilde\alpha\in p\Oc_{D_0}$ if and only if $o_p(n) = o_p(aD_0t + \Nm(\tilde\alpha)) \ge 1$.
    Therefore for $p \mid t$, we have
    $   \delta_p(n, \tilde\alpha)    =      1    $ if and only if $o_p(n) \le 1$ for $p \mid D_{0, t}$ and $o_p(n) = 0$ for $p \mid t$ but $p \nmid D_0$.
    Otherwise, $\delta_p(n, \tilde\alpha) = 0$.
    Note the condition on $o_p(n)$ can be succinctly written as $\gcd(n, p^N) \mid D_{0, t}$ for any $N \ge 2$. 
Putting these together with \eqref{eq:dpp-1} and \eqref{eq:dpp-2} gives us
    $$
  \rho^{(D\ell)}(n/A^2)    \delta'_\ell(n, \tilde\alpha)
    \prod_{p\mid D \text{ prime}, p \neq \ell}\delta_p(n, \tilde\alpha)
    =
      \sigma(\gcd(n, \tfrac{D_0'}{\gcd(\ell, D_0')}))
      \begin{cases}
\sum_{j \ge o_\ell(D_0')} \rho_{}(\frac{n/A^2}{\ell^j})        
      & \text{if } \ell \nmid t,\\
      \rho(\frac{n/A^2}{\ell \gcd(\ell, D_{0, t})}) & \text{if } \ell \mid t,
    \end{cases}
    $$
    for $\ell \neq 3$ and $\gcd(n^{(\ell)}, t^2) \mid D_{0, t}$.

    Suppose $\Diff(-n/(D_0t), \Nc_1) = S(D_0, -na)$ contains another prime $\ell'\neq\ell$.
If $\ell'$ is inert,    then $2 \nmid o_{\ell'}(n)$ and $\rho(n/(A^2\ell^j)) = 0$.
    So the equality above still holds in this case.
    To ensure that the right hand side also vanishes when $\ell'$ is ramified, it suffices to multiply it by the characteristic function that 
     $[\Ac]$ is in the genus representing $-\kappa_\ell \cdot n$.
     Then summing over $\Ac \in S_0$ proves \eqref{eq:res1} and \eqref{eq:res2}.     
     The case for $\ell = 3$ in \eqref{eq:res3} follows similarly using \eqref{eq:s3contr}.
\end{proof}

Similarly, one has the following analogue of Corollary~\ref{cor:main}.
\begin{corollary}\label{cor:main2}
    Let $D_i = D_0 t_i^2$ with $D_0 \equiv 1 \bmod8$ fundamental, and all primes dividing $t = t_1t_2 > 1$ are non-split in $\kay$. 
Then the prime factors of the resultant of the minimal polynomials of $f(\mathfrak{A}_i)^{24/s}$ are no larger than~$D_{0}t$.
\end{corollary}

\begin{example}
    Take $D_{1}=-7$ and $D_{2}=-5^{2}7$ and $s=1$. Then $D_{0}=-7$, $t=5$, the defining polynomial associated with $\mathcal{O}_{D_{1}}$ is $P_{1}(X)=X-1$, and the defining polynomial associated with $\mathcal{O}_{D_{2}}$ is
    \begin{align*}
      P_{2}(X)&=  X^6-45771X^5+1046370975X^4\\
      &\quad+293236687600X^3+23843150292975X^2-273301922603526X+1,
    \end{align*}
    so the resultant ${R}(P_{1},P_{2})$ of $P_{1}(X)$ and $P_{2}(X)$ is $-3^{14}\cdot5\cdot7^2\cdot19^3\cdot31$. In what follows, we use~\eqref{eq:res} to compute the prime factorization for the resultant. Under the choices of $D_{1}$, $D_{2}$ and $s$, the formula~\eqref{eq:res} yields that for $\ell\ne5$ inert or ramified in $\mathbb{Q}[\sqrt{-7}]$,
    \begin{align*}
        {\rm ord}_{\ell}({R}(P_{1},P_{2}))&=  \sum_{\substack{n, \tn\in\mathbb{N}\\n + \tn = 35\\ (n, 5) = 1\\4|n\tilde{n}}}
\sigma(\gcd(n, 7/\ell))
  \sum_{\substack{AB = 2\\j \ge 1}}
  \rho^{}\lp \frac{n}{A^2\ell^{j}}\rp
\rho(\tn/B^2)\\
&=\sum_{\substack{1\leq k\leq8\\k\ne5}}\sigma(\gcd(4k,7/\ell))\sum_{j\geq1}\rho(k/\ell^{j})\rho(35-4k)\\
&\quad+\sum_{\substack{1\leq k\leq 8\\k\ne 5}}\sigma(\gcd(35-4k,7/\ell))\sum_{j\geq1}\rho((35-4k)/\ell^{j})\rho(k).
    \end{align*}
Note by Corollary~\ref{cor:main2} that $\ell$ cannot exceed~35. So the primes $\ell\ne5$ that may have contributions to the resultant are $3$, $7$, $13$, $17$, $19$, $31$. Applying the formula above, one can compute and find that
$$
{\rm ord}_{3}({R}(P_{1},P_{2}))=14,\quad {\rm ord}_{7}({R}(P_{1},P_{2}))=2,\quad {\rm ord}_{13}({R}(P_{1},P_{2}))=0,
$$
$$
{\rm ord}_{17}({R}(P_{1},P_{2}))=0,\quad {\rm ord}_{19}({R}(P_{1},P_{2}))=3,\quad {\rm ord}_{31}({R}(P_{1},P_{2}))=1.
$$
    Finally, for $\ell=5$, by~\eqref{eq:res2} one can first find that
    \begin{align*}
      {\rm ord}_{5}({R}(P_{1},P_{2}))
      &=
        \sum_{\substack{n+\tilde{n}=35\\\tilde{n}\geq0, n\geq1\\4|n\tilde{n}}}
      \sigma(\gcd(n,7))\sum_{AB=2}\rho^{}(n/(\ell A^{2}))
    \rho(\tn/B^2)\\
&= \rho^{}(15/5 )     \rho(20/4)
+  \rho^{}(20/5)     \rho(15)
+  2\rho^{}(35/5)     \rho(0) = 1.
    \end{align*}
Putting all the data above gives the prime factorization for ${R}(P_{1},P_{2})$ that matches the numerical result.

\end{example}



\section{Local calculation} \label{sect:local}
\subsection{Values of local Whittaker functions}
\label{subsec:Whitt}
In this subsection, we use different notation for convenience of future application of the basic formulas. Let $F$ be a finite field extension of $\Q_p$ with a uniformizer $\pi$ and $q=|\pi|^{-1}$, and assume $p\ne 2$. Let $0 \ne \Delta \in \OO_F$, and let $E =F(\sqrt\Delta)$ be the quadratic extension of $F$ ($E =F + F$ if $\Delta$ is square) with order $\OO_\Delta=\OO_F + \sqrt \Delta \OO_F$. Let $0\ne \kappa \in\OO_F$, and let $\mathcal M=\OO_\Delta$ with quadratic form $Q_\kappa(x) =\kappa x \bar x$. So its dual $\mathcal M' =\frac{1}{\kappa \sqrt\Delta} \mathcal M$. 
Our goal is to give an explicit formula for the local density integral for $\mu \in  \mathcal M'/\mathcal M$.
\begin{equation}
\label{eq:Wm}
W_{m}(s, \mu) =
W_{m}(s, \mu; \mathcal{M})
:=\int_{F} \int_{\mu+ \OO_{\Delta}} \psi(b\kappa x \bar x)  \psi(-mb)|a(wn(b))|^s  dxdb.
\end{equation}
Here  $\psi$ is an unramified non-trivial additive character of $F$, and $\hbox{Vol}(\OO_\Delta, dx) = |\Delta|^{\frac{1}2}=q^{-\frac{1}2 o(\Delta)}$.  This integral was calculated carefully in \cite[Appendix]{YYY} when $\mathcal O_\Delta =\mathcal O_E$ i.e., $o(\Delta) =0$  or $1$.  Similar calculation gives the following three propositions we need in this paper. We leave the detail to the reader.

Let $\chi_{-\pi}$ be the quadratic character of $F^\times$ associated to the quadratic extension $F(\sqrt{-\pi})/F$. 
For a number $b \in F^\times$, write $b=b_0 \pi^{o(b)}$ with $b_0 \in \OO_F^\times$ and  $o(b) \in \Z$. For $b=0$, we set $b_0=1$ and $o(b) =\infty$.
Notice that (for $b \ne 0$)
\begin{equation}
\chi_{-\pi}(b) = (\frac{b_0}{\pi} )=\pm 1
\end{equation}
depending on whether  $b_0$ is a square modulo $\pi$ (equivalently whether $b_0$ is a square in $F$).

\begin{proposition} \label{prop4.5} One has $W_m(s, 0) =0$ unless $m \in \OO_F$. Assume $m \in \OO_F$.  Write $X= q^{-s}$.

\begin{enumerate}
\item When  $ o(m) -o(\kappa) < 0$, we have
$$
|\Delta|^{-\frac{1}2} W_m(s, 0)= (1-X) \sum_{0 \le n \le o(m)} (qX)^n.
$$

\item When  $0 \le o(m)-o(\kappa) < o(\Delta)$ , we have
\begin{align*}
|\Delta|^{-\frac{1}2} W_m(s, 0)&= (1-X) \sum_{0 \le n < o(\kappa)} (qX)^n
 +  (1-X^2)(qX)^{o(\kappa)} \sum_{0 \le n < \frac{1}2 o(m/\kappa)} (qX^2)^n
\\
&\quad + \begin{cases}
q^{\frac{1}2o(m\kappa)}X^{o(m)} (1+ \chi_{-\pi}(m \kappa) X) &\fff o(m/\kappa) \hbox{ is even},
   \\
    0 &\fff o(m/\kappa) \hbox{ is odd} .
    \end{cases}
\end{align*}

\item When  $ o(m/\kappa) \ge o(\Delta), m \neq 0$, and $o(\Delta)$ is even ($E/F$ is unramified), we have
\begin{align*}
|\Delta|^{-\frac{1}2} W_m(s, 0)&= (1-X) \sum_{0 \le n < o(\kappa)} (qX)^n
 +  (1-X^2)(qX)^{o(\kappa)} \sum_{0 \le n < \frac{1}2 o(\Delta)} (qX^2)^n
 \\
 & \quad + \frac{q^{\frac{1}2o(\kappa^2\Delta)} X^{o(\kappa)}}{L(s+1, \chi_\Delta)} \sum_{o(\Delta) \le n \le o(m/\kappa)} (\chi_{\Delta}(\pi) X)^n.
 \end{align*}

\item When  $o(m/\kappa) \ge o(\Delta), m \neq 0$, and $o(\Delta)$ is odd ($E/F$ is ramified), we have
\begin{align*}
|\Delta|^{-\frac{1}2} W_m(s, 0)
 &=(1-X) \sum_{0 \le n < o(\kappa)} (qX)^n+ (1-X^2)(qX)^{o(\kappa)} \sum_{0 \le n < \frac{1}2 (o(\Delta)-1)} (qX^2)^n
 \\
 &\quad  +q^{\frac{1}2 (o(\kappa^2 \Delta) -1)} X^{o(\kappa \Delta)-1}
   \left[ 1 +\chi_{\Delta}(m\kappa) X^{o(\frac{m}{\kappa\Delta})+2}\right].
\end{align*}

\item Suppose  $m = 0$, then we have 
\begin{align*}
  |\Delta|^{-\frac{1}2} W_0(s, 0)
  &= (1-X) \sum_{0 \le n < o(\kappa)} (qX)^n
 +  (1-X^2)(qX)^{o(\kappa)} \sum_{0 \le n < \lfloor \frac{1}2 o(\Delta)\rfloor} (qX^2)^n
 \\
  & \quad +
q^{\lfloor\frac{1}2o(\kappa^2\Delta)\rfloor} X^{o(\kappa) + 2\lfloor o( \Delta)/2 \rfloor}
    \frac{L(s, \chi_\Delta)}{L(s+1, \chi_\Delta)} 
 \end{align*}

\end{enumerate}
\end{proposition}
When  $o(\Delta) =0 $ or $1$, this gives \cite[Propositions 3.6.2 and 3.6.3]{HY12} (see also \cite[Propositions 5.7 and 5.8]{YYY}.

\begin{remark}
 The formula (2) in Proposition \ref{prop4.5} does not contradict to the fact that if $m$ is in the Diff set, then  the value should be zero.
Indeed,  assuming (2), then  $W_m(0, 0) \ne 0$ if and only if  $o(m/\kappa)$ is even and  $\chi_{-\pi}(m \kappa) =(\frac{m_0 \kappa_0}{\pi}) =1$, i.e., $m \kappa$ is a square in $F$.  
So $W_m(0, 0) \ne 0$ implies that the quadratic space $(V=F(\sqrt{\Delta}), Q(x) =\kappa x \bar x)$ represents $m$. 
This example also shows that even when $(V=F(\sqrt{\Delta}), Q(x) =\kappa x \bar x)$ represents $m$,  $W_m(0, 0)$ could still be zero. 

\end{remark}
Now assume that $\mu=\mu_1  + \mu_2 \sqrt\Delta \in  \mathcal M'-\mathcal M$ with  $\mu_1 \in \frac{1}\kappa \OO_F$ and $\mu_2 \in  \frac{1}{\kappa\Delta}\OO_F$. Denote
$$
o(\mu) =\begin{cases}
   \min(o(\mu_1),  o(\mu_2 \Delta)) &\fff \mu_1,  \mu_2 \notin \OO_F,
   \\
   o(\mu_1) &\fff \mu_2 \in \OO_F,
   \\
   o(\mu_2 \Delta) &\fff \mu_1 \in  \OO_F  .
  \end{cases}
$$
   Let
\begin{equation}
\alpha(\mu, m) = -\kappa \mu \bar\mu + m,  \quad \hbox{and }  o(\mu, m) = o(\alpha(\mu, m)/\kappa).
\end{equation}
Both depend on the choice of $\mu \in  \mathcal M'$, not just its image in $\mathcal M'/\mathcal M$.

\begin{proposition} \label{prop4.6} Assume  $\mu_1=0$ and  $o(\Delta \mu_2) \ge 0$. Then  $W_m(s, \mu)=0$ unless $\alpha(\mu, m) \in \OO_F$.  Assume $\alpha(\mu, m) \in \OO_F$, then the following holds.

\begin{enumerate}
\item When  $o(\mu, m) <0$, we have
$$
|\Delta|^{-\frac{1}2} W_m(s, \mu) =(1-X) \sum_{0 \le n \le o(\mu, m)+o(\kappa)} (qX)^n.
$$

\item When  $ 0\le o(\mu, m) < o(\mu)= o(\mu_2\Delta)$, we have
\begin{align*}
|\Delta|^{-\frac{1}2} W_m(s, \mu) &=(1-X) \sum_{0 \le n < o(\kappa)} (qX)^n
 +(1-X^2)(qX)^{o(\kappa)} \sum_{0 \le n <\frac{1}2 o(\mu, m) } (qX^2)^n
\\
 &\quad+\begin{cases}
    0 &\fff o(\mu, m) \hbox{ is odd},
    \\
    q^{o(\kappa)+ \frac{1}2o(\mu, m)} X^{o(\mu, m)} (1+ \chi_{-\pi} (\kappa  \alpha(\mu, m)) X)
       &\fff o(\mu, m)  \hbox{ is even}.
    \end{cases}
\end{align*}

\item  When  $o(\mu, m) \ge o(\mu)$, we have
\begin{align*}
|\Delta|^{-\frac{1}2} W_m(s, \mu)
&= (1-X) \sum_{0 \le n \le o(\kappa)} (qX)^n
\\
 &\quad  +(1-X^2)(qX)^{o(\kappa)} \sum_{0 \le n <[\frac{1}2 o(\mu)] } (qX^2)^n
  + (qX^2)^{[\frac{o(\mu)}{2}]} (qX)^{o(\kappa)},
  \end{align*}
Here $[a]$ means the integer part of $a$.
\end{enumerate}

\end{proposition}

Notice that Proposition \ref{prop4.6}(1) (2) is pretty much the same as Proposition \ref{prop4.5}(1) (2) with $m$ replaced by $\alpha(\mu, m)$.

\begin{proposition} \label{prop4.7}  Assume $\mu_1 \notin \OO_F$ or $o(\Delta \mu_2) <0$. Then

\begin{enumerate}
\item  when  $o(\mu, m) < o(\mu)$, we have
$$
|\Delta|^{-\frac{1}2} W_m(s, \mu) =(1-X) \sum_{0 \le n \le o(\mu, m)+o(\kappa)} (qX)^n.
$$

\item When $o(\mu, m) \ge o(\mu) $, we have
$$
|\Delta|^{-\frac{1}2} W_m(s, \mu) =(1-X) \sum_{0 \le n < o(\kappa \mu)} (qX)^n  + (qX)^{o(\kappa\mu)}.
$$

\end{enumerate}

\end{proposition}


\subsection{Local Splitting of $\phi_\dd$}
Recall that $\tilde\af, \af$ are the $\OO_{4D_0}$-ideals defined in \eqref{eq:tildea}.
From Prop.\ \ref{prop:Lattice}, we have
$$
\tilde \af \oplus \af \stackrel{i_\dd}\cong \Pc \oplus \Nc \subset L_\dd. 
$$
 For each prime $p$, we have the local isometry
  $$
  i_{\dd, p}:= i_{\dd} \otimes \Z_p :(\taf \times \af) \otimes \Z_p\to L_{\dd} \otimes \Z_p,
  $$
from \eqref{eq:mapi}, using which we can define
  \begin{equation}
    \label{eq:tphi}
    \tilde{\phi}_{p\dd_p} := \phi_{\dd_p, p} \circ i_{\dd_p, p} \in S(\taf \otimes \Z_p) \otimes S(\af \otimes \Z_p).
  \end{equation}
  For example, if $p = 2$ and $\dd = 12$, then we would have $\tilde{\phi}_8$ and $\tilde{\phi}_9$ at the places $2$ and $3$ respectively. 
  We can now decompose this purely local Schwartz function $\tilde{\phi}_{\dd_p p}$ into a sum of tensor products in $S(\taf \otimes \Z_p) \otimes S(\af \otimes \Z_p)$ explicitly.
Note that at the places $p \nmid 6D$, the finite quadratic modules $\Pc'/\Pc$, $\Nc'/\Nc$ and $L'_\dd/L_\dd$ are all trivial.
So we just need to consider the primes $p \mid 6D$. 


\subsubsection{$p=2$} 
\label{subsec:p2}

For any $\dd \mid 24$, $m \in \Qb_{\ge 0}$ and $\alpha \in \frac{1}{2\sqrt{D}} \taf$ satisfying $\frac{t}{a}\Nm(\alpha) = 1 - m$, define the quantity
The goal is to calculate the sum of these values over $\dd_p \mid s_p$ for some $s \mid 24$.

We can specialize the local calculations in \cite{LY} to the cases here.
Let $\dd \mid 24$ be fixed and consider the small CM point $Z = (z_1, z_2)$ with $z_i$ chosen as in Lemma \ref{lem:smallCM}.
First, since $D_1, D_2< 0$ are discriminants satisfying \eqref{eqd}, we know that $2$ splits in $\kay_{D_0}$. Therefore, we will fix $\delta_0 \in \Z_2$ satisfying $\delta_0^2 = D_0$ and
\begin{equation}
  \label{eq:zr}
  \zr_i := t_i a_i^{-1} \frac{b + \delta_0}{2} \in 1 + 2\Z_2,~
    \overline{\zr_i} := t_i a_i^{-1} \frac{b - \delta_0}{2} \in 2\Z_2.
  \end{equation}
   for $i = 1, 2$ by Lemma \ref{lem:smallCM}.
We now identify the following $\Q_2$-vector spaces \footnote{To ease notation,  we will use $(x_i) = (x_1, \dots, x_n) \in F^n$ to denote column vectors of size $n$ with entries in a field $F$.}

\begin{equation}
  \label{eq:kay2}
  \begin{split}
    \kay_{D_0} \otimes_{} \Z_2 &\stackrel{\cong}{\rightarrow} \Q_2^2\\
    (\alpha + \beta \sqrt{D_0})\otimes 1 & \mapsto (\alpha + \beta \delta_0, \alpha - \beta \delta_0) .
  \end{split}
\end{equation}
This becomes an isometry with respect to the quadratic forms $C\cdot \Nm(\lambda)$ for $\lambda \in \kay_{D_0} \otimes \Q_2$ and $C \cdot (x_1 x_2)$ for $(x_1, x_2) \in \Q_2^2$, for any $C \in \Q$.
The complex conjugation in $\kay_{D_0}$ induces the automorphism on
\begin{equation}
  \label{eq:barQ2}
\overline{ (x_1, x_2) } := (x_2, x_1)
\end{equation}
for $ (x_1, x_2 ) \in \Q_2^2$.
Under this isometry, the map $i_\dd^{-1}: V_\dd \stackrel{\cong}{\to} \kay_{D_0} \times \kay_{D_0}$ gives rise to
\begin{equation}
  \label{eq:iotad2}
  \begin{split}
    i_{\dd, 2}^{-1}: V_\dd \otimes \Q_2 &\to \Q_2^4\\
    \smat{\mu_3}{\mu_1}{\mu_4}{\mu_2} \otimes 1 &\mapsto P_2\cdot (\mu_i),
  \end{split}
\end{equation}
where $i_{\dd, 2} := i_{\dd} \otimes \Q_2$ and
\begin{equation}
  \label{eq:P}
  P_2 := -\frac{a}{t \delta_0}
  \begin{pmatrix}
    1 & - \overline{\zr_1} & \zr_2 & -\overline{\zr_1}\zr_2 \\
    -1 &  {\zr_1} & -\overline{\zr_2} & \zr_1\overline{\zr_2} \\
    -1 & \overline{\zr_1} & -\overline{\zr_2} & \overline{\zr_1\zr_2} \\
    1 & -{\zr_1} & {\zr_2} & -{\zr_1\zr_2} \\
  \end{pmatrix} \in \SL_4(\Z_2).
\end{equation}
and its reduction moduluo 2, denoted by $\overline{P_2}$, is $
\lp \begin{smallmatrix}
  1 & 0 & 1 & 0 \\
  1 & 1 & 0 & 0 \\
  1 & 0 & 0 & 0 \\
  1 & 1 & 1 & 1
\end{smallmatrix}\rp \in \SL_4(\Fb_2)$ with $\Fb_2 := \Z_2/2\Z_2$.
\begin{remark}
  \label{rmk:LYcompare}
  The identification $V_\dd \otimes \Q_2 \cong \Q_2^4$ given in (6.14) in \cite{LY} is
  $$
  \pmat{\mu_3}{\mu_1}{\mu_4}{\mu_2} \mapsto -\frac{t \delta_0}{a}
\lp  \begin{smallmatrix}
    1 & & & \\
     & -1 & & \\
     & &-1 & \\
     & & & 1\\
  \end{smallmatrix}\rp
  P_2\cdot (\mu_i).
  $$
\end{remark}

For $n \in \N$ and any subspace $\Vb \subset (\Z_2/2\Z_2)^n = \Fb_2^n$, define the sublattice of the lattice $\Z_2^n$ in the vector space $\Q_2^n$
\begin{equation}
  \label{eq:MV}
  M_\Vb := \{x \in \Z_2^n: x \bmod 2\Z_2 \in \Vb\}.
\end{equation}
Notice that $M_{\Fb_2^n} = \Z_2^n$.
Also for any $A \in M_n(\Z_2)$ and $\Vb \subset \Fb_2^n,\Vb' \subset \Fb_2^{n'}$, we have
\begin{equation}
  \label{eq:MVtrans}
 A \cdot  M_\Vb   = M_{
\overline{A}      \cdot \Vb },~
  M_{\Vb} \times   M_{\Vb'} = M_{\Vb \times \Vb'}.
\end{equation}
The subspaces $\Vb$ that we will be particularly interested is the following family of subspaces of codimension 1
\begin{equation}
  \label{eq:Vbn}
  \begin{split}
    \Vb_n &:=  \left\{(\alpha_i) \in \Fb_2^n: \sum_{i = 1}^n \alpha_i = 0\right\} =
        (1, \dots, 1)^\perp \subset \Fb_2^n,\\
  \end{split}
\end{equation}
where $\perp$ is with respect to dot product on $\Fb_2^n$.

Under the isometry in \eqref{eq:kay2}, the $\Z$-modules $\tilde{\af}, \af, \af_i, \bfrak_i$ defined in \eqref{eq:af} for $i = 0, 1, 2$ becomes
\begin{equation}
  \label{eq:ideals2}
  \af \otimes \Z_2
= \bfrak_1 \otimes \Z_2
= \tilde{\af} \otimes \Z_2
\cong M_{\Vb_2} \subset \Z_2^2 \cong \af_i \otimes \Z_2. 
\end{equation}
We can now characterize the 2-part of $\Pc_\dd \oplus \Nc_\dd \subset L_\dd$ in the manner below.
\begin{proposition}
  \label{prop:PNL2}
  Under the isometry in \eqref{eq:iotad2},
   the lattices $\Pc_\dd \oplus \Nc_\dd \subset L_\dd \subset \frac{1}{\dd} L_\dd \subset L_\dd' \subset \Pc_\dd' \oplus \Nc_\dd'$ becomes
  \begin{equation}
    \label{eq:lattchain}
    M_{\Vb_2}^2 \subset
    M_{ \Vb_{4}} \subset
    \frac{1}{ \dd_2}     M_{\Vb_{4}} \subset
 \frac{1}{2\dd_2} M_{\Vb_{4}^\perp} \subset \frac{1}{2\dd_2} M_{\Vb_2}^2,
\end{equation}
with the quadratic form $\tilde{Q}_\dd((y_i)) := \frac{\dd t}{a} (y_1 y_2 - y_3 y_4)$ for $(y_i) \in \Q_2^4$.
\end{proposition}
\begin{proof}
  It is clear from \eqref{eq:ideals2} that the images of $\Pc_\dd \oplus \Nc_\dd$ and $\Pc_\dd' \oplus \Nc_\dd'$ under the isometry in \eqref{eq:iotad2} are $M^2_{\Vb_2}$ and $  \frac{1}{2\dd_2} M_{\Vb_2}^2$ respectively.
  For $L_\dd$, we have $L_\dd \otimes_{} \Z_2 = M_{(0, 0, 0, 1)^\perp} \subset \Z_2^4$.
  Therefore, its image under the isometry in \eqref{eq:iotad2} is
  $$
P_2 \cdot M_{(0, 0, 0, 1)^\perp}  = M_{\overline{P_2} \cdot (0,0,0,1)^\perp} = M_{\Vb_4}.
$$
where we have used \eqref{eq:MVtrans}.
  Similarly, $L_\dd' \otimes_{} \Z_2 = \frac{1}{2\dd_2} M_{\langle (1, 0, 0, 0) \rangle}$ and
  $$
  P_2 \cdot M_{\langle (1, 0, 0, 0) \rangle}  = M_{\overline{P_2} \cdot \langle (1,0,0,0) \rangle}
  = M_{\langle(1,1,1,1) \rangle}
  = M_{\Vb_4^\perp}.
$$
    This finishes the proof.
  \end{proof}

Now, we want to apply the map in \eqref{eq:iotad2} to express the $\phi_{\dd, 2}\in S(L_\dd \otimes \Z_2)$ as a Schwartz function
\begin{equation}
  \label{eq:tphi}
  \tilde{\phi}_{\dd, 2} := \phi_{\dd, 2} \circ i_{\dd, 2} \in S(M_{\Vb_4}) \subset
 S(M_{\Vb_2}^2) =   S(M_{\Vb_2}) \otimes S(M_{\Vb_2}).
\end{equation}
The function $\phi_{\dd, 2}$ is constant on $2M_2(\Z_2) \subset L_\dd \otimes \Z_2$ with support on $\frac{1}{\dd_2} M_2(\Z_2) \supset \frac{1}{\dd}L_\dd \otimes \Z_2$.
Using \eqref{eq:lattchain} and $\frac{1}{\dd_2}M_{\Vb_4} \subset \frac1{\dd_2}\Zb_2^4$, we see that 
\begin{equation}
    \label{eq:supp}
    \mathrm{supp}(\phi_{\dd, 2}) \subset \frac1{\dd_2} \taf_0 \oplus\af_0 = (\taf_0 \oplus\af_0)'. 
\end{equation}
The restriction of $\iota_{\dd, 2}^{-1}$ to $S(L_\dd \otimes \Z_2)$ only depends on $P_2 \otimes \Z/16\Z$, where $P_2$ is the matrix in \eqref{eq:P}.
By the choice of $z_j$ in Lemma \ref{lem:smallCM}, we know that
$$
P_2 \equiv
-\delta_0^{-1}
  \begin{pmatrix}
    1 & - \overline{\zr} & \zr & -\overline{\zr}\zr \\
    -1 &  {\zr} & -\overline{\zr} & \zr\overline{\zr} \\
    -1 & \overline{\zr} & -\overline{\zr} & \overline{\zr^2} \\
    1 & -{\zr} & {\zr} & -{\zr^2} \\
  \end{pmatrix} \bmod{2^4 \Z_2}
  $$
with $\zr := \frac{1 + \delta_0}{2} \in \Z_2^\times,~ \overline{\zr} := \frac{1 - \delta_0}{2}\in 2\Z_2$.
We can now specialize Lemma 6.3 in \cite{LY} and apply Remark \ref{rmk:LYcompare} to give a precise expression of $\tilde{\phi}_{\dd, 2}$.

\begin{lemma}
  \label{lemma:decompose2}
  In the notations above, we have
\begin{equation}
  \label{eq:phidfactor}
  \tilde{\phi}_{2\dd_2}  =
  \begin{cases}
\cha({M_{\Vb_4}})
,& \dd_2 = 1,\\
\phi_0\otimes( \phi_{2}^3 - \phi_2^1)
+ (\phi_{2}^1 - \phi_2^3)\otimes  \phi_0
& \dd_2 = 2,\\
  2\lp (\phi_0 + \phi_{2}^1 - \phi_{2}^3)\otimes (\phi_{4}^7 - \phi_4^3)
+ (\phi_{4}^1 - \phi_4^5) \otimes (\phi_0 - \phi_{2}^1 + \phi_2^3)\rp,& \dd_2 = 4,\\
4 \lp
 \begin{split}
&    (\phi_0 + \phi_{2}^3 - \phi_2^1)\otimes( \phi_{8}^{15} - \phi_{8}^7) + (\phi_{8}^1 - \phi_{8}^9)\otimes( \phi_0 + \phi_{2}^1 - \phi_2^3)  \\
&    + (\phi_{4}^7 - \phi_4^3)\otimes( \phi_{8}^{11} - \phi_{8}^3)  + (\phi_{8}^5 - \phi_{8}^{13})\otimes( \phi_{4}^1 - \phi_4^5)
 \end{split}
\rp,
    & \dd_2 = 8,
  \end{cases}
\end{equation}
  where for $m = 2, 4, 8$ and $r \in (\Zb/2m\Zb)^\times$
  \begin{align*}
    \phi_0 &:= \cha \lp M_{\Vb_2}\rp - \cha \lp M_{\Vb_2} + (1, 0)\rp
=    \phi_{(0, 0)} +
    \phi_{(1, 1)}
    - \phi_{(1, 0)}
    - \phi_{(0, 1)},\\
  \phi_{m}^r &:= \sum_{a, b \in \Zb/2m\Z,~ ab = r} \phi_{\tfrac{1}{m}(a, b)}
\end{align*}
are elements in $S((2\Zb_2)^2)$.
\end{lemma}

\begin{proof}
  By Remark \ref{rmk:LYcompare} and the discussions before the lemma, we see that the $\tilde{\phi}_{2\dd_2}$ here  is the same as the $\tilde{\phi}_{\dd_2, 2}$  in Lemma 6.3 in \cite{LY} with $\delta = \delta_0 \delta_0 = D_0 \equiv 1 \bmod{8}$. This then gives us the result after taking into consideration the sign difference between the quadratic forms here and in Lemma 6.3 loc.\ cit. Also note that the local lattice here is not scaled by~2 as in \cite{LY}, where the Schwartz functions are in $S(\Zb_2^2)$.
\end{proof}

\begin{proposition}
  \label{prop:d2contr}
  For any $\dd_2 \mid 8$,  the quantity $  \delta_2(m, \alpha; \dd_2)$ defined in \eqref{eq:delta6}
    is given by
  \begin{align*}
    \delta_2(m, \alpha; 1)
    &=
      \begin{cases}
        1 & \text{ if }o_2(m) = 0 \text{ and } 2 \mid \alpha,\\
        o_2(m) - 1& \text{ if }o_2(m) \ge 2 \text{ and } m \neq 0,\\
        1 & \text{ if } m = 0,\\
      \end{cases}\\
    \delta_2(m, \alpha; 2)
    &=
      \begin{cases}
        1 & \text{ if } 4 \mid \alpha,\\
        \pm 1 & \text{ if } m \equiv 3 \pm 2 \bmod{8} \text{ and } 2 \| \alpha,\\
        1 & \text{ if } o_2(m) = 2,\\
        o_2(m) - 5& \text{ if }o_2(m) \ge 3 \text{ and } m \neq 0,\\
        1 & \text{ if } m = 0,
      \end{cases}
    \end{align*}
    \begin{align*}
    \delta_2( m, \alpha; 4)
    &= 2\cdot
      \begin{cases}
        1 & \text{ if } 8 \mid \alpha,\\
        \pm 1 & \text{ if } m \equiv 9 \pm 8 \bmod{32} \text{ and } 4 \| \alpha,\\
        \mp 1 & \text{ if } m \equiv 5 \bmod{8} \text{ and } \tm \equiv \pm 4 \bmod{16}, \\
        \mp 1 & \text{ if } m \equiv \pm 4 \bmod{16},\\
        1 & \text{ if }o_2(m) = 4,\\
        o_2(m) - 7& \text{ if }o_2(m) \ge 5 \text{ and } m \neq 0,\\
        1 & \text{ if } m = 0,\\
      \end{cases}\\
      \delta_2(m, \alpha; 8)
    &= 4\cdot
      \begin{cases}
        1 & \text{ if } 16 \mid \alpha,\\
        \pm 1 & \text{ if } m \equiv 33 \pm 32 \bmod 128 \text{ and } 8 \| \alpha,\\
        \pm 1 & \text{ if } m \equiv 1 \pm 16 \bmod{64} \text{ and } 4 \| \alpha,\\
        \mp 1 & \text{ if } m \equiv 5 \bmod{8} \text{ and } \tm \equiv 20 \pm 8 \bmod{32}, \\
        \pm 1 & \text{ if } m \equiv 4 \pm 8 \bmod{32},\\
        \mp 1 & \text{ if } m \equiv \pm 16 \bmod{64},\\
        1 & \text{ if }o_2(m) = 6,\\
        o_2(m) - 9& \text{ if }o_2(m) \ge 7 \text{ and } m \neq 0,\\
        1 & \text{ if } m = 0.
      \end{cases}
    \end{align*}
\end{proposition}

\begin{proof}
We will give the details for $\dd_2 = 2$ and leave the rest of the cases to the reader.
First, we can apply Lemma \ref{lemma:decompose2} to write
  $$
  \phi_{8}\lp \frac{\alpha}{\dd}, \frac{\mu_2}{\dd}\rp
  =
  \phi_0(\alpha/\dd) (\phi^3_2(\mu_2/\dd) - \phi^1_2(\mu_2/\dd)) +
(\phi^1_2(\alpha/\dd) - \phi^3_2(\alpha/\dd))  \phi_0(\mu_2/\dd) .
$$
Using Table 4 in \cite{LY}, we have
\begin{align*}
 & \sum_{\mu_2 \in \Zb_2^2/2M_{\Vb_2}}
  (\phi^3_2(\mu_2/\dd) - \phi^1_2(\mu_2/\dd)) \frac{W^*_{m/\dd, 2}(0, \mu_2/\dd)}{\gamma(W_2)}\\
    &= \sum_{\mu_2 \in (\Zb_2/4\Zb_2)^2} (\phi_{\frac{1}{2}(1, 3)}(\frac{\mu_2}{2\dd_3})
    + \phi_{\frac{1}{2}(3, 1)}(\frac{\mu_2}{2\dd_3})
      - \phi_{\frac{1}{2}(1, 1)}(\frac{\mu_2}{2\dd_3}) - \phi_{\frac{1}{2}(3, 3)}(\frac{\mu_2}{2\dd_3}))
      \frac{W^*_{m/\dd, 2}(0, \mu_2/\dd)}{\gamma(W_2)}\\
  &=  \frac{W^*_{m/\dd, 2}(0, \frac{1}{2}(1, 3)) + W^*_{m/\dd, 2}(0, \frac{1}{2}(3, 1))}{\gamma(W_2)} = 1
\end{align*}
when $o_2(\tm/\dd) \ge 1$, and 0 otherwise. Notice the scaling factor between the quadratic forms here and in \cite{LY} means that the coset $\frac{1}{4}(a, a^{-1} \delta)$ becomes $\frac{1}{2}(a, -a)$ for us.
Now, if we denote $\alpha_2 := \alpha \otimes \Zb_2 \in (2^{-1} \Zb_2)^2$, then $\phi_0(\alpha/\dd)$ becomes
\begin{align*}
  \phi_0(\alpha/\dd) =
  \begin{cases}
    1& \text{ if } \alpha_2 \equiv (0, 0) \text{ or } (2, 2) \bmod{(4\Zb_2)^2},\\
    -1& \text{ if } \alpha_2 \equiv (2, 0) \text{ or } (0, 2) \bmod{(4\Zb_2)^2}.
  \end{cases}
\end{align*}
which also implies $\alpha \in \frac{1}{\sqrt{D}}\taf$.
For the first case, $\alpha_2 \equiv (0, 0) \bmod{(4\Zb_2)^2}$ is equivalent to $4\mid \alpha$.
The condition $\alpha_2 \equiv (2, 2) \bmod{(4\Zb_2)^2}$ is equivalent to $2 \| \alpha $ and $o_2(\Nm(\alpha)) = 2$. The second case happens if and only if $2 \| \alpha$ and $o_2(\tm) = o_2( \Nm(\alpha)) \ge 3$.
Using $\tm + m = 1$, we see that
$$
o_2(\tm) = 2 \Leftrightarrow m \equiv 5 \bmod{8},~
o_2(\tm) \ge 3 \Leftrightarrow m \equiv 1 \bmod{8}.
$$
Putting these together gives us the result when $\dd_2 = 2$ and $o_2(m) = 0$.

Similarly, we can apply Table 1 in \cite{LY} to calculate the  contribution of the other term
\begin{align*}
    &\sum_{\mu_2 \in  \Zb_2^2/2M_{\Vb_2}}
  \phi_0(\mu_2/\dd) \frac{W^*_{m/\dd, 2}(0, \mu_2/\dd)}{\gamma(W_2)} \\
  &= \sum_{\mu_2 \in \Zb_2^2/(4\Zb_2)^2} (\phi_{(0, 0)}(\frac{\mu_2}{2\dd_3}) + \phi_{(1, 1)}(\frac{\mu_2}{2\dd_3}) - \phi_{(0, 1)}(\frac{\mu_2}{2\dd_3} )- \phi_{(1, 0)}(\frac{\mu_2}{2\dd_3}))
    \frac{W^*_{m/\dd, 2}(0, \mu_2/\dd)}{\gamma(W_2)}\\
  &=
    \begin{cases}
1      & \text{ if } o_2(m/\dd) = 1 (\Leftrightarrow \tm \equiv 5 \bmod{8}),\\
o_2(m/\dd) - 4       & \text{ if } o_2(m/\dd) \ge 2 (\Leftrightarrow \tm \equiv 1 \bmod{8}).
    \end{cases}
\end{align*}
On the other hand, $\Nm(\alpha) = \frac{a}{t} \tm \equiv 1 \bmod{4}$ implies that
$$
\phi^1_2(\alpha/\dd) - \phi^3_2(\alpha/\dd) = \phi^1_2(\alpha/\dd) = (\phi_{\frac{1}{2}(1, 1)} + \phi_{\frac{1}{2}(3, 3)})(\alpha/\dd) = 1.
$$
Putting this together gives  the result for $\dd_2 = 2$ and $o_2(m) \ge 2$.
\end{proof}

Now we will add the contributions of $\delta_2(\alpha, m; \dd_2)$ over all $\dd_2 \mid s_2$ for some $s_2 \mid 8$.
\begin{proposition}
  \label{prop:s2contr}
  Let $s_2 \mid 8$ and $\delta_2(m, \alpha; \dd_2)$ be as in the previous proposition. Then for all $\alpha \in \frac{1}{\sqrt{D}} \taf$ with $\tm := \frac{t}{a}\Nm(\alpha) = 1 - m$, we have
  \begin{equation}
    \label{eq:d2id}
    \sum_{\dd_2 \mid s_2} \delta_2( m, \alpha; \dd_2) =
    s_2 
          \sum_{\substack{r_2 \mid s_2\\ \frac{m\tm}{4r_2^2} \equiv 3 \bmod{\frac{s_2}{r_2}}}}
      \sum_{AB = 2r_2}   \rho_2\lp \frac{m}{A^2} \rp
      \mathds{1}_{\taf}\lp \frac{\sqrt{D} \alpha}{B} \rp
  \end{equation}
  where $\mathds{1}_{\taf}$ is the characteristic function of $\taf \subset \kay$.
\end{proposition}
\begin{remark}
  Since $m + \tm = 1$, at most one of the terms in the sum over $AB = 2r_2$ is non-zero.
\end{remark}

\begin{proof}
  The proof comes from directly applying Prop.\ \ref{prop:d2contr}.
  If neither of $m$ and $\tm$ is 2-integral, then both sides vanishes identically. Otherwise, $m$ and $\tm$ will have different parity and one term in the summand on the RHS of \eqref{eq:d2id} will automatically vanishes.

  When $o_2(m) = 0$, the first term always vanishes and the second term becomes
  $$
s_2     \sum_{\substack{r_2 \mid s_2\\ \frac{m\tm}{4r_2^2} \equiv 3 \bmod{\frac{s_2}{r_2}}}}
\mathds{1}_{2}( \frac{\alpha}{2r_2}).
$$
Notice that at most one term in the sum above does not vanish as $\Nm(\alpha) \equiv \tm \bmod{16}$.
  For $s_2 = 1$, this matches with the LHS.
  For $s_2 = 2$, the LHS is non-zero in the following cases
  $$
  \delta_2(1, m , \alpha) + \delta_2(2, m, \alpha) =
  \begin{cases}
    2 & \text{ if } 4 \mid \alpha (\Leftrightarrow m \equiv 1 \bmod{16}), \\
2 &  \text{ if } 2 \| \alpha \text{ and } m \equiv 5 \bmod{8},
  \end{cases}
  $$
  which also matches with the RHS.
  The cases with $s_2 = 4, 8$ are similar and we omit the details here.

  When $o_2(m) \ge 1$, then $\tm$ is odd and the second term always vanishes and the first term becomes
  $$
s_2     \sum_{\substack{r_2 \mid s_2\\ \frac{m\tm}{4r_2^2} \equiv 3 \bmod{\frac{s_2}{r_2}}}}
\rho_2(\frac{m}{4r_2^2}),
$$
where again at most one term does not vanish.
For $s_2 = 1$, this equals to $o_2(m/4) + 1 = \delta_2(1, m, \alpha)$.
For $s_2 = 2$, the LHS becomes
  $$
  \delta_2( m , \alpha; 1) + \delta_2( m, \alpha; 2) =
  \begin{cases}
    2 & \text{ if } o_2(m) = 2,\\
2(o_2(m) - 3) &  \text{ if } o_2(m) \ge 3,
  \end{cases}
  $$
  which again matches with the RHS.
  The cases with $s_2 = 4, 8$ are similar and we omit the details here.
  \end{proof}

\subsubsection{$p = 3$.}

The contribution from the $3$-part is summarized in the following concise proposition.
\begin{proposition}
  \label{prop:s3contr}
  Let $s_3 \mid 3$ and $s_3' := \gcd(s_3, 3^{1 - (\tfrac{D}3)}) \mid s_3$. 
    Then for all $\alpha \in \frac{1}{\sqrt{D}} \taf$ with $\tm := \frac{t}{a}\Nm(\alpha) = 1 - m$, we have
  \begin{equation}
    \label{eq:s3contr}
    \begin{split}
s_3^{-1}          \sum_{\dd_3 \mid s_3} \delta_3(m, \alpha; \dd_3) &=
      \sum_{\substack{s'_3 \mid r_3 \mid s_3\\ \frac{m\tm}{r_3^2} \equiv 1 \bmod{\frac{s_3}{r_3}}}}
          \sum_{AB = r_3}   \rho_3\lp \frac{m}{A^2} \rp \mathds{1}_{\taf}\lp \frac{\sqrt{D} \alpha}{B} \rp, \\
s_3^{-1}          \sum_{\dd_3 \mid s_3} \delta_3'(m, \alpha; \dd_3) &=
\rho'_{3s_3}(m) :=
    \begin{cases}
    \frac{o_3(m/s_3) + 1}{2}& \text{ if } 2 \nmid o_3(m) \ge 1, m \neq 0 \text{ and } \lp \frac{D}{3}\rp = -1,\\
      0&\text{ otherwise.}
          \end{cases}
    \end{split}
  \end{equation}
\end{proposition}

We will spend the rest of the section proving this result.
If $s_3 = 1$, this is clear as 
\begin{equation}
  \label{eq:d3=1}
  \begin{split}
    \delta_3(m, \alpha; 1) &= \rho_3(m)~\text{ if } 2 \mid o_3(m) \ge 0 ,\\
    \delta_3'(m, \alpha; 1) &= \frac{o_3(m) + 1}{2} = \sum_{j \in \Nb} \rho_3(m/3^j) ~\text{ if } 2 \nmid o_3(m) \ge 1 \text{ and } \lp \frac{D}{3}\rp = -1.
  \end{split}
\end{equation}
Notice that $\rho'_{3s_3}(m)$ defined in \eqref{eq:s3contr} becomes $\rho'_3(m)$ in \eqref{eq:rho'} when $s_3 = 1$.
Also, the first equation in \eqref{eq:d3=1} is always 1 for $m = 0$, and therefore the second equation is only relevant for $m \neq 0$.

If $\dd_3 = 3$, recall that $\tilde{\phi}_{3\dd_3} = \tilde{\phi}_9 \in S(\OO_D) \otimes S(\OO_D)$ is the Schwartz function defined in \eqref{eq:tphi}.
The quadratic form on $(\alpha, \mu)\in (\kay_D \otimes \Zb_3)^2$ is given by $\tilde{Q}(\alpha, \mu) = \frac{3 t}{a} (\Nm(\alpha) - \Nm(\mu))$.
Since $3 \mid t-a$, we can work with the quadratic form $3(\Nm(\alpha) - \Nm(\mu))$.

There are now 2 cases to consider.
\begin{itemize}
\item
$\lp \frac{D_i}{3} \rp = 1$
\item
$\lp \frac{D_i}{3} \rp = -1$
\end{itemize}
The first case is similar to the case $p = 2$ considered above.
We have $\OO_D = \Zb_3^2$ and fix $\delta_0 \in 1 + 3\Zb_3$ such that
$$
\delta_0^2 = D_0.
$$
%
Then the map $i_{\dd, 3}$ modulo $3\Z_3$, in which case it is the same as given by $z_j' := \frac{b + \sqrt{D_0}}{2}$.
We can now apply Lemma 6.6 in \cite{LY} to these $z_j'$ and express $\tilde{\phi}_{\dd, 3}$ explicitly as follows.

\begin{lemma}
  \label{lemma:decompose3a}
  In the notations above, we have
$$
\tilde{\phi}_{9} = \phi_0 \otimes\phi_{-1} + \phi_{1}\otimes \phi_0 + 2 \phi_{-1}\otimes \phi_{1},
$$
where
\begin{align*}
  \phi_0 &:= 3\phi_{(0, 0)} - \sum_{a, b \in \Z/3\Z,~ ab = 0} \phi_{\tfrac{1}{3}(a, b)},~
\phi_{\pm 1} := \phi_{\tfrac{1}{3}(1, \pm 1)} + \phi_{\tfrac{1}{3}(2, \pm 2)}
\end{align*}
are in $S(\Zb_3^2)$.
\end{lemma}

\begin{lemma}
  \label{lemma:delta3}
  The quantity $  \delta_3(m, \alpha; 3)$
  defined in \eqref{eq:delta6}
  is given by
  \begin{align*}
    \delta_3(m, \alpha; 3)
    &=
      \begin{cases}
        2&\text{if } o_3(m(1-m)) = 0,\\
        2(o_3(m) - 2)&\text{if } o_3(m) = o_3(m(1-m)) \ge 1 \text{ and } m \neq 0,\\
        2 & \text{if } m = 0,\\
        3 \mathds{1}_3(\alpha/3) - 1&\text{if } o_3(m(1-m)) > o_3(m) = 0.
      \end{cases}
  \end{align*}
  Otherwise it is 0.
\end{lemma}

\begin{lemma}
  \label{lem:s3contr1}
  Let $s_3 \mid 3$ and $\delta_3(m, \alpha; \dd_3)$ be as in the previous proposition. Then for all $\alpha \in \frac{1}{\sqrt{D}} \taf$ with $\tm := \frac{t}{a}\Nm(\alpha) = 1 - m$, we have
  \begin{equation}
    \label{eq:d3id}
    \sum_{\dd_3 \mid s_3} \delta_3(m, \alpha; \dd_3) =
s_3     \sum_{\substack{r_3 \mid s_3\\ \frac{m\tm}{r_3^2} \equiv 1 \bmod{\frac{s_3}{r_3}}}}
     \rho_3(\frac{m}{r_3^2}) + \mathds{1}_{3}( \frac{r_3}{3}) \mathds{1}_{3}( \frac{\alpha}{r_3}) .
  \end{equation}
\end{lemma}

Now for the second case, we need to calculate both $\delta_3(m, \alpha;\dd_3)$ and $\delta_3'(m, \alpha;\dd_3)$.
The procedure is the same as before, and we can again use the calculations in \cite{LY}.

\begin{lemma}  In the notations above, we have
$$
\tilde{\phi}_{9} =
2 \sum_{\mu \in S_{1}} \phi_\mu \otimes \phi_{0}  +
2 \sum_{\mu \in S_{-1}} \phi_0 \otimes \phi_{\mu}  -
 \sum_{\mu_1 \in S_{-1}, \mu_2 \in S_{1}} \phi_{\mu_1} \otimes \phi_{\mu_2}
$$
where $S_j :=  \{ \mu \in \tfrac{1}{3}{\mathcal{O}_D}/{\mathcal{O}_D}: {3}{}\Nm(\mu) \equiv \tfrac{j}{3} \bmod{\Zb_3}\}$ for $j = \pm 1$.
\end{lemma}

The value of $\delta_3$ and $\delta'_3$ can be calculated similarly.

\begin{lemma}
  The quantities $  \delta_3(m, \alpha; 3)$
  defined in \eqref{eq:delta6} is 0 except in the following cases
  \begin{align*}
        \delta_3(m, \alpha; 3)
    &=
      \begin{cases}
        -1&\text{if } o_3(m(1-m)) = 0,\\
        2&\text{if } 2 \mid o_3(m(1-m)) \ge 1.
      \end{cases}
  \end{align*}
  If $2 \nmid o_3(m) \ge 1$ and $m \neq 0$, then $  \delta_3(m, \alpha; 3) = 0 $  and $    \delta_3'(m, \alpha; 3)
    = o_3(m) - \frac{ 1}{2}
$.
\end{lemma}
\begin{proposition}
  \label{prop:s3contr2}
  Let $s_3 \mid 3$.
  Then we have
  \begin{equation}
    \label{eq:d3id1}
    \begin{split}
          \sum_{\dd_3 \mid s_3} \delta_3(m, \alpha; \dd_3) &=
    s_3   \rho_3(\frac{m(1-m)}{s_3^2}), \\
          \sum_{\dd_3 \mid s_3} \delta_3'(m, \alpha; \dd_3) &=
    s_3 \frac{o_3(m/s_3) + 1}{2},~ \text{ if } 2\nmid o_3(m) \ge 1 \text{ and } m \neq 0.
    \end{split}
  \end{equation}
\end{proposition}

\subsubsection{$p \mid D$.}
\label{subsec:pmidD}


In this case, we have $\dd_p = 1$. Recall $A$ is the finite quadratic module defined in \eqref{eq:Aid}. Using the identification there, we have the following result.  
\begin{lemma}
\label{lemma:phi-split-D}
  In the notations above, for any prime $p \mid D$, we have $\dd_p = 1$ and 
  \begin{equation}
    \label{eq:phi1p1}
    \tilde{\phi}_{p\dd_{p}} = \tilde\phi_p =
    \begin{cases}
      \sum_{\mu \in \frac{1}{\sqrt D}\OO_{D_0, p}/\OO_{D_0, p}} \phi_{-\mu} \otimes \phi_{\mu},& \text{ if } p \nmid t_1,\\
      \sum_{\mu \in \frac{1}{\sqrt D}\OO_{D_0, p}/\OO_{D_0, p}} \phi_{\overline{\mu}} \otimes \phi_{\mu},& \text{ if } p \nmid t_2.
    \end{cases}
  \end{equation}
\end{lemma}
\begin{remark}
    If $p \nmid t = t_1t_2$, then $p \mid D_0$ and $- \mu = \overline{\mu}$. 
\end{remark}
\begin{proof}  Write 
$$
\mu= \frac{\alpha + \beta \sqrt{D_0}}{ \sqrt{D}}, \hbox{ and  }\tilde\mu = \frac{\tilde\alpha + \tilde\beta \sqrt{D_0}}{\sqrt{D}},
$$
with the numerators are in $\mathcal O_{D_0, p}$. Then 
$$
i^-(\mu) =-\frac{1}{tD_0}
 \left[ (\alpha+ \beta \sqrt{D_0}) \omega_1(z_1, z_2)  + \overline{(\alpha+ \beta \sqrt{D_0}) \omega_1(z_1, z_2)} \right]
 $$
 and similarly for $i^+(\tilde\mu)$. 
 Directly calculation shows that $
\iota^+(\tilde\mu) + \iota^-(\mu) \in  L_p
$ implies  (looking at three of the four entries of the matrices
$$
\alpha +\tilde{\alpha}  \in t D_0 \mathbb Z_p
$$
and 
$$
\beta +\tilde\beta \in t_2 \mathbb Z_p, 
\hbox{ and } 
\beta -\tilde\beta \in   t_1 \mathbb Z_p.
$$
When $p\nmid t_2$, the condition $\beta +\tilde\beta \in t_2 \mathbb Z_p$ is automatic, and we have $ \tilde\mu - \bar\mu \in  \mathcal O_{D_0, p}$.
When  $p \nmid t_1$, the condition $\beta -\tilde\beta \in t_1 \mathbb Z_p$ is automatic, and we have $ \tilde\mu + \mu \in  \mathcal O_{D_0, p}$.
Since $(t_1, t_2) =1$, this covers all cases.  Finally, it is easy to check directly that the converse is also true. 
\end{proof}
Now the quantity $\delta_p(n, \tilde{\alpha})$ defined in \eqref{eq:deltaD} is given explicitly as follows.
\begin{proposition}
  \label{prop:dpcontr}
  Let $n \in \Qb_{\ge 0}$ and $\tilde{\alpha} = \tilde{\alpha}_1 + \sqrt{D_0} \tilde{\alpha}_2  \in
  \taf_0$ be any element satisfying $\frac{1}{a}\Nm(\tilde{\alpha}) = -D_0t - n$ for an $\mathcal{O}_{D_0}$-fractional ideal
  $\taf_0$  co-prime to $D$ with norm $a$.
  Recall that $\epsilon = \epsilon_{\kay/\Q}$ is the Dirichlet character associated to $\kay = \Qb(\sqrt{D_0})/\Qb$. 
  Denote $r:= o_p(D_0 t), r_0 := o_p(D_0),\alpha := -\frac{\tilde\alpha_1^2 + an}{aD_0t} = 1- \frac{\tilde\alpha_2^2}{at}$ and $\Oc := \Oc_{D_0} \otimes \Zb_p = \Zb_p[\sqrt{D_0}]$. 

  \begin{description}
  \item[Case (i)] Suppose $n \neq 0$.
When and $p \not\in \Diff(-n/(D_0t), \Nc_1)$, the quantity $  \delta_p(n, \tilde{\alpha})$ defined in \eqref{eq:deltaD} is given by
\begin{align*}
  \begin{cases}
p^{(o_p(n)-r_0)/2}(    1+ \chi_p(\tfrac{an}{D_0}))
    &\text{if }   \tilde{\alpha} \in \sqrt{D}\Oc, 2r -r_0 \le o_p(n) < 2r,     \text{ and } 2 \mid o_p(n) - r_0,\\
\frac{p^{r-r_0/2}}{L(1, \epsilon)}(o_p(n) - 2r + 1)
  &\text{if }  \tilde{\alpha} \in \sqrt{D}\Oc,
2r \le o_p(n) \text{ and }\epsilon(p) = 1,\\
    \frac{p^{r-\lceil r_0/2 \rceil}}{L(1, \epsilon)} (2+\epsilon(p))
  &\text{if } \tilde{\alpha} \in \sqrt{D}\Oc, 2r \le o_p(n)  \text{ and }\epsilon(p) \neq 1,\\
     p^{o_p(\alpha)}(1 + \chi_p(-at \alpha)) &\text{if } 
0 \le o_p(\alpha)< o_p(\tilde\alpha_1) < r = r_0 \text{ and } 2 \mid o_p(\alpha),\\
     p^{\lfloor o_p(\tilde{\alpha}_1)/2\rfloor} &\text{if } 
                                                  o_p(\tilde\alpha_1) \le \min(o_p(\alpha),r-1) \text{ and }r = r_0,\\
    1
    & \text{if } \min(o_p(\alpha_1), o_p(\alpha_2)) = 0 \text{ and } r_0 < r.
\end{cases}  
\end{align*}
  Otherwise, it is  zero. 
  In particular, when $p \mid t$, we have
  \begin{equation}
    \label{eq:dp-1}
    \delta_p(n, \tilde\alpha)
    =
      \begin{cases}
        1 & \text{if }\tilde\alpha\not\in p\Oc,\\
        0 & \text{otherwise.}
      \end{cases}
  \end{equation}
When $r=r_0=1$, we have
  \begin{equation}
    \label{eq:dp-2}
    \delta_p(n, \tilde\alpha)
    =
    \begin{cases}
        2 & \text{if }o_p(n) \ge 1,\\      
        1 & \text{if }o_p(n) = 0,\\
        0 & \text{otherwise.}
      \end{cases}
  \end{equation}
When $p \in \Diff(-n/(D_0t), \Nc_1)$, the quantity $  \delta_p'(n, \tilde{\alpha})$ defined in \eqref{eq:deltaD} is given by
  \begin{align*}
     \begin{cases}
\frac{p^{o_p(n)-r+1}-1}{p-1}
    &\text{if }   \tilde{\alpha} \in \sqrt{D}\Oc \text{ and } r \le o_p(n) < 2r-r_0\\
           \frac{p^{\lceil (o_p(n)-r_0+1)/2\rceil} + p^{\lfloor (o_p(n)-r_0+1)/2\rfloor}
           -p^{r-r_0} -1}{p-1}
    &\text{if }   \tilde{\alpha} \in \sqrt{D}\Oc \text{ and } 2r -r_0 \le o_p(n) < 2r,\\
\frac{2p^{r-\lceil r_0/2\rceil}-p^{r-r_0}-1}{p-1}&\\
+           \frac{(2+\epsilon(p))p^{r-\lceil r_0/2 \rceil}}{2L(1, \epsilon)}(o_p(n) - 2\lfloor \tfrac{r_0}2 \rfloor + r_0 - 2r + 1)
  &\text{if }  \tilde{\alpha} \in \sqrt{D}\Oc
\text{ and }2r \le o_p(n)\\
1
    &\text{if } \tilde\alpha \in t\Oc \backslash \sqrt{D}\Oc \text{ and } r_0 < r,\\
2\frac{p^{\lceil (o_p(\alpha))/2\rceil}-1    }{p-1}
+ \frac{1 + (-1)^{o_p(\alpha)}}2 p^{(o_p(\alpha))/2}
    &\text{if } o_p(\alpha) < o_p(\tilde\alpha_1)<r = r_0,\\
           1
    & \text{if } 0 < \min(o_p(\alpha_1), o_p(\alpha_2)) < r-r_0 .
         \end{cases}
  \end{align*}
  Otherwise, it is  zero. 
  In particular, for any $p \mid t$, we have
  \begin{equation}
    \label{eq:dpp-1}
    \delta_p'(n, \tilde\alpha)
    =
      \begin{cases}
        1 & \text{if } \tilde\alpha \in p\Oc,\\
        0 & \text{otherwise.}
      \end{cases}
    \end{equation}
When $r=r_0=1$, we have
    \begin{equation}
    \label{eq:dpp-2}
    \delta'_p(n, \tilde\alpha)
    =
    \begin{cases}
        o_p(n) & \text{if }o_p(n) \ge 1,\\      
        0 & \text{otherwise.}
      \end{cases}
    \end{equation}
  \item[Case (ii)] Suppose $n = 0$ and $r_0 \le 1$. 
      If $r = r_0 = 1$, then we have
  \begin{equation}
    \label{eq:dp0a}
    \delta_p(0, \tilde\alpha) =
    \begin{cases}
      1 & \text{if }\tilde\alpha \in \sqrt{D}\Oc,\\
      0 & \text{otherwise.}
    \end{cases}
  \end{equation}
  If $r > r_0$, i.e.\ $p \mid t$, then we have
  \begin{equation}
    \label{eq:dp0b}
    \delta_p(0, \tilde\alpha) = 0,~
    \delta'_p(0, \tilde\alpha) = \frac1{L_p(0, \epsilon)} = 1 - \epsilon(p).
  \end{equation}

  \end{description}    
  \end{proposition}



\begin{remark}
    \label{rmk:chip}
    If $2 \nmid o_p(D_0)$ and $2 \mid o_p(n)-r_0$,
    then 
    $$
    \chi_p(an/D_0) 
    =
    \begin{cases}
        1 & p \not\in \Diff(-n/(D_0t), \Nc_1),\\
        -1 & p \in \Diff(-n/(D_0t), \Nc_1).  
    \end{cases}
    $$
\end{remark}
\begin{proof}
  Applying Lemma \ref{lemma:phi-split-D}, in addition to the definitions of $\delta_p(n, \tilde\alpha)$ and $\delta'_p(n, \tilde\alpha)$ in \eqref{eq:deltas} and \eqref{eq:deltaD}, we have for $n > 0$
\begin{align*}
    \delta_p(n, \tilde\alpha)
&= p^{r-r_0/2} \delta_p(0, -n/(D_0 t), \tilde\alpha/\sqrt{D}; 1)
= p^{r_0/2} 
\begin{cases}
    W_{-n/(D_0 t)}(0, -\tilde\alpha/\sqrt{D}) & p \nmid t_1,\\
    W_{-n/(D_0 t)}(0, -\overline{\tilde\alpha}/\sqrt{D}) & p \nmid t_2,
\end{cases}\\
\delta'_p(n, \tilde\alpha)
&= \frac{p^{r-r_0/2}}{\log p} \delta'_p(0, -n/(D_0 t), \tilde\alpha/\sqrt{D}; 1)
= \frac{p^{r_0/2} }{\log p}
\begin{cases}
    W'_{-n/(D_0 t)}(0, -\tilde\alpha/\sqrt{D}) & p \nmid t_1,\\
    W'_{-n/(D_0 t)}(0, -\overline{\tilde\alpha}/\sqrt{D}) & p \nmid t_2.
\end{cases}
\end{align*}
    We now apply the results in section \ref{subsec:Whitt} with 
    $$
    F = \Qb_p, \Delta = D_0, \kappa = -t/a, m = -\frac{n}{D_0t},
    \mu \in \Oc -\frac{\tilde\alpha_1 \pm \tilde\alpha_2 \sqrt{D_0}}{\sqrt{D}}, -\pi = p, X = p^{-s}.
    $$ 
     We write $o$ for $o_p$ and $\Diff = \Diff(-n/(D_0t), \Nc_1) = \Diff(-an, (\tilde\af_0, \Nm))$ for convenience. 
    
     Proposition \ref{prop4.5} has 5 cases, and is applicable when $\alpha \in \sqrt{D}\Oc$.
     In case (1), we have $0 \le o(n/(D_0t)) < o(t)$, and $\delta_p=0$, whereas $\delta'_p=\frac{p^{o_p(n)-r+1}-1}{p-1}$ as in case I for $\delta_p'$. 
     In case (2), $o(t) \le o(n/(D_0t)) < r_0$.
     If $p \notin \Diff$, then $\delta_p$ is non-zero when $2\mid o(n/D_0)$, which gives case I for $\delta_p$.
     Otherwise, $\chi_p(m\kappa) = \chi_p(na/D_0)$ is always $-1$, and we obtain case II for $\delta_p$ and $\delta'_p$.
     Case (3) and (4) happends for $o(n/(D_0t^2)) \ge r_0$, and yield cases II and III for $\delta_p$, and case III for $\delta'_p$.

     Proposition \ref{prop4.6} has 3 cases, and is applicable when $\tilde\alpha \in t\Oc\backslash \sqrt{D}\Oc$.
     To apply it, we take $\mu = \mu_2\sqrt{D_0} = -\frac{\tilde\alpha_1}{\sqrt{D}}$. Then
     $$
     \alpha 
     =     \alpha(\mu, m) = \kappa \mu_2^2 D_0 + m = - \frac{\tilde\alpha_1^2/a + n}{D_0t} = 1 - \frac{\tilde\alpha_2^2}{at} \in \Zb_p. 
     $$
     In case (1), the condition $0 \le o(\alpha) < o(t)$ implies $p \mid t$ and $o(\alpha) = o(1-\tilde\alpha_2^2/at) = 0$, since $\tilde\alpha_2^2/t \in t\Zb_p$.
     This proves case IV for $\delta'_p$.
     In case (2), the condition $o(t) \le o(\alpha) < o(\tilde\alpha_1) = o(\Delta\mu_2) + o(\kappa)$ implies that
     $$
0 = o(1) = o\lp \alpha + \frac{\tilde\alpha_2^2}{at}\rp \ge \min(o(\alpha), o(\tilde\alpha_2^2/t)) \ge o(t),
$$
i.e.,\ $p \nmid t$.
When $p\not\in\Diff$, we obtain case IV for $\delta_p$. 
When $p\in\Diff$, there does not exist $\beta_1, \beta_2 \in\Zb_p$ such that $\frac{an + \beta_2^2}{D_0} = \beta_1^2$, hence $\chi_p(-at\alpha) = \chi_p((\tilde\alpha_1^2+an)/D_0) = -1$. Then specializing the  to $X =1$ gives us $\delta_p = 0$ and case V for $\delta'_p$.
In case (3), we have $o(\alpha/t) \ge o(\tilde\alpha_1/t) \ge 0$.
So $o((1-\tilde\alpha_2^2/(at))/t) \ge 0$, which implies $p \nmid t$ since $\tilde\alpha_2 \in t\Zb_p$. Furthermore,  $-an \equiv \tilde\alpha_1^2 \bmod p$ implies that $-an$ is a square in $\Zb_p$ by Hensel's lemma. In particular, it is always  a norm from $\Q_p(\sqrt{D_0})$, and $p$ is never in $\Diff$. Specializing $X = 1$ gives us case V for $\delta_p$. 

    Proposition \ref{prop4.7} has 2 cases and is applicable when $\tilde\alpha \not\in t\Oc$, which happens only if $p \mid t$.
    We set $\mu = \tilde\alpha/\sqrt{D}$, and $\alpha(\mu, m) = 1$. 
    Checking case by case, we always have $o(\mu) = \min(o(\tilde\alpha_1/t), o(\tilde\alpha_2/t))$.
    In case (1), $-o(t) < o(\mu)$ implies that $\tilde\alpha_i \in p\Zb_p$ for $i = 1, 2$. Setting $X=1$ gives us case VI of $\delta'_p$. 
    In case (2), we have similarly $\tilde\alpha \not\in p\Oc$, and obtain case VI of $\delta_p$.

    Suppose $p \mid t$, i.e.\ $r > r_0$. For $\tilde\alpha \in \sqrt{D}\Oc$, we have equivalently $o(\tilde\alpha_1) \ge r, o(\tilde\alpha_2) \ge r-r_0$. So $2r \ge 2r-r_0 > r$ and
    $$
o(n) = o(\tilde\alpha_1^2 - D_0 \tilde\alpha_2^2 + aD_0t) = o(D_0t) = r. 
    $$
    So only the last case for $\delta_p$ and cases I, IV, VI for $\delta'_p$ are applicable. Those give us equations \eqref{eq:dp-1} and  \eqref{eq:dpp-1}.

    When $r=r_0=1$, the condition $o_p(n) \ge 1$ directly implies that $\tilde\alpha \in \sqrt{D}\Oc$. Simplifying the expressions in $\delta_p$ and $\delta'_p$ gives us equations \eqref{eq:dp-2}, and     \eqref{eq:dpp-2}.
    This proves Case (i).

    For Case (ii), when $r = r_0 = 1$, we have $\tilde\alpha \in \sqrt{D} \Oc$ and can apply Proposition \ref{prop4.5} (5) to obtain \eqref{eq:dp0a}.
    When $r > r_0$, then $o(\Nm(\tilde\alpha)) = o(D_0t) = r < 2r-r_0$, which means $\tilde\alpha \not\in \sqrt{D} \Oc$ and Proposition \ref{prop4.5} is no longer applicable.
    So    we use Propositions \ref{prop4.6} and Propositions \ref{prop4.7}. 
    As    $o(\mu, 0) = o((\tilde\alpha_1^2 - D_0 \tilde\alpha_2^2)/D) = -o(t) < 0$, case (1) in those two Propositions are applicable, and they give us \eqref{eq:dp0b}. 
\end{proof}



\end{document}